\documentclass[11pt]{amsart}
\usepackage[english,francais]{babel}
\usepackage[latin1]{inputenc}
\usepackage{amssymb}
\usepackage{amsmath}
\usepackage{amsfonts}
\usepackage{amsthm}
\usepackage{enumerate}
\usepackage{color}

\numberwithin{equation}{section}
\newtheorem{hypothesis}{Hypothesis}
\newtheorem{theorem}{Theorem}[section]
\newtheorem{lemma}{Lemma}[section]
\newtheorem{proposition}{Proposition}[section]
\newtheorem{corollary}{Corollary}[section]

\newtheorem{remark}{Remark}[section]

\def\R{\mathbb{R}}
\def\Z{\mathbb{Z}}

\def\K{\mathbb{K}}
\def\S{\mathcal{S}}

\def\N{\mathbb{N}}

\def\F{\mathcal{F}}

\def\T{\mathbb{T}}

\def\supp{\mathop{\rm supp}\nolimits}
\newcommand\cro[1]{\langle #1 \rangle}

\def\ut{\widetilde{u}}
\def\zt{\widetilde{z}}
\def\wt{\widetilde{w}}

\begin{document}
\selectlanguage{english}
\title[  ]{\bf Improvement of the energy method for strongly non resonant dispersive equations and applications}
\author{Luc Molinet}
\address{Laboratoire de Math\'ematiques et Physique Th\'eorique\\
Universt\'e Francois Rabelais Tours,\\
CNRS UMR 7350- F\'ed\'eration Denis Poisson\\
Parc Grandmont, 37200 Tours, France}
\author{St\'ephane Vento}
\address{Université Paris 13\\ Sorbonne Paris Cité, LAGA, CNRS ( UMR 7539)\\  99, avenue Jean-Baptiste Clément\\ F-93 430 Villetaneuse\\ France }
\begin{abstract}
In this paper we propose a new approach to prove the local well-posedness of the Cauchy problem associated with strongly non resonant dispersive equations. As an example we obtain unconditional well-posedness of the Cauchy problem below $ H^1 $  for a large class of  one-dimensional dispersive equations with a dispersion that is greater or equal to the one of the Benjamin-Ono equation. Since this is done without using a gauge transform, this enables us to prove strong convergence results for  solutions  of viscous versions of these equations towards the purely dispersive solutions.
\end{abstract}
\subjclass[2010]{}
\date{}
\maketitle
\section{Introduction}
The Cauchy problem associated with dispersive equations with derivative nonlinearity has been extensively studied since the eighties.
The first results  were obtained by using energy methods that did not make use of the dispersive effects (see for instance \cite{kato} and  \cite{abfs}). These methods were restricted to regular initial data ($ s>d/2 $ were $d \ge 1 $ is the spatial dimension) and only ensured the continuity of the solution-map. At the end of the eighties, Kenig, Ponce and Vega proved new dispersive estimates that enable them to lower the regularity requirement on the initial data
 (see for instance \cite{KPV1}, \cite{KPV2}, \cite{Po1}). They even succeed to obtain local well-posedness for a large class of dispersive equations by a fixed point argument in a suitable Banach space related to linear dispersive estimates. Then in the early nineties, Bourgain introduced the now so-called Bourgain's spaces where one can solve by a fixed point argument  a wide class of dispersive equations with very rough initial data (\cite{Bourgain1993}, \cite{Bourgain1993KP}). It is worth noticing that, since the nonlinearity of these equations is in general algebraic,  the fixed point argument ensures the real analyticity of the solution-map. Now, in the early 2000's,
 Molinet, Saut and Tzvetkov \cite{MST} noticed that a large class of "weakly" dispersive equations, including in particular the Benjamin-Ono equation, cannot be solved by a fixed point argument for initial data in any Sobolev spaces $ H^s $.
  This obstruction is due to bad interactions between high frequencies and very low frequencies.  Since then, roughly speaking,  two approaches have been developed to lower the regularity requirement for such equations. The first one is the so called gauge method. This consists in introducing a nonlinear gauge transform  of the solution  that solved an equation with less bad interactions than the original one. This method was proved to be very efficient to obtain the lowest regularity index for solving canonical equations (see \cite{Tao}, \cite{IK}, \cite{BP}, \cite{MP} for the BO equation and \cite{HIKK} for dispersive generalized BO equation)  but has the disadvantage to behave very bad with respect to perturbation of the equation. The second one consists in improving the dispersive estimates by localizing it on space frequency depending time intervals and then mixing it  with classical energy estimates. This type of method was first introduced by Koch and Tzvetkov  \cite{KT1} (see also \cite{KK} for some
 improvements)
  in the framework of Strichartz's spaces and then by Koch and Tataru \cite{KTa} (see also \cite{IKT}) in the framework of Bourgain's spaces. It is less efficient to get the best regularity index but it is surely more flexible with respect to perturbation of the equation.

In this paper we propose a new approach to derive local and global well-posedness results for dispersive equations that do not exhibit too strong resonances. This approach combines classical energy estimates with Bourgain's type estimates on an interval of time that does not depend on the space frequency.  
 Here, we will apply this method to prove  unconditional  local well-posedness results on both $\R$ and $\T$ without the use of a gauge transform for a large class of one-dimensional quadratic dispersive equations with a dispersion between the one of the Benjamin-Ono equation and the KdV equation. This class contains in particular the equations with  pure power dispersion that read
\begin{equation}\label{gv}
u_t + \partial_x D_x^{\alpha} u  + u u_x=0
\end{equation}
with $ \alpha\in [1,2] $. \\
The principle of the method is particularly simple in the regular case $ s>1/2 $. We start with the classical space frequency localized energy estimate
\begin{equation}\label{energy-est}
\|P_N u\|^2_{L^\infty_T H^s} \lesssim \|P_N u_0\|_{H^s}^2 + \sup_{t\in ]0,T[}  \langle N\rangle^{2s} \Bigl| \int_0^t\int_{\R}   \partial_x P_N(u^2) P_N u \Bigr|
\end{equation}
obtained by projecting the equation on frequencies of order $N$ and taking the inner product with $J^s_x u$. Note that the second term in the RHS of (\ref{energy-est}) is easily controlled (after summing in $N$) by $\|u\|_{L^ \infty_TH^s}^3$ for $s>3/2$. This is the main point in the standard energy method that lead to LWP in $H^s$, $s>3/2$. In order to take into account the dispersive effects of the equation, we will decompose the three factors in the integral term into dyadic pieces for the modulation variables and use the Bourgain's spaces $X^{s,b}$ in a non conventional way. Actually, it is known that standard bilinear estimates in $ X^{s,b}$-spaces with $b=1/2+$ fail, for equation \eqref{gv}, for any $s\in \R$ as soon as $\alpha<2$. On the other hand, as noticed in
 \cite{MN},  it is easy to deduce from the equation that a solution $ u\in L^\infty(0,T;H^s) $ to \eqref{bo} has to belong to the space $ X^{s-1,1}_T $. This means that, if we accept to lose a few spatial derivatives on the solution, then we may gain some regularity in the modulation variable. This is particularly profitable when the equation enjoys a strongly non resonance relation as (\ref{resonance}). Actually,
 this formally allows to estimate the second term in (\ref{energy-est}) at the desired level. However,  this term  involves a multiplication  by   $ 1_{]0,t[}$  and it is well-known that such multiplication is not bounded in  $ X^{s-1,1}$. To overcome this difficulty we  decompose this function into two parts. A high frequency part that will be very small in $ L^1_T $ and a low-frequency part that will have good properties with respect to multiplication  with high modulation functions in $ X^{s-1,1} $. This decomposition will depend on the space frequency localization of the three functions that appear in the trilinear term.

\subsection{Presentation of the results}
In this paper we consider the dispersive equation
\begin{equation}\label{bo}
u_t + L_{\alpha+1} u  + u u_x=0
\end{equation}
where $x\in \R $ or $ \T$, $ u=u(t,x) $ is a real-valued function and  the linear operator $ L_{\alpha+1} $satisfies the following hypothesis.
\begin{hypothesis}	\label{hyp1}
We assume that $ L_{\alpha+1} $ is the Fourier multiplier operator by $ i p_{\alpha+1} $ where  $ p_{\alpha+1} $ is a real-valued  odd function satisfying : \\
 For any $ (\xi_1,\xi_2)\in \R^2 $ with $ |\xi_1| \gg 1 $ and any $ 0<\lambda\ll 1 $ it holds
\begin{equation}\label{hyp2}
\lambda^{\alpha+1} |\Omega(\lambda^{-1}\xi_1, \lambda^{-1} \xi_2) |
\sim |\xi|_{\min} |\xi|_{\max}^\alpha \, ,
\end{equation}
where
$$
\Omega(\xi_1,\xi_2):=p_{\alpha+1}(\xi_1+\xi_2)-p_{\alpha+1}(\xi_1)-p_{\alpha+1}(\xi_2) \, ,
$$
 $ |\xi|_{\min}:=\min(|\xi_1|,|\xi_2|, |\xi_1+\xi_2|) $ and $ |\xi|_{\max}:=\max(|\xi_1|,|\xi_2|, |\xi_1+\xi_2|) $.
 \end{hypothesis}
 \begin{remark}\label{remark-ex}
 We will see in Lemma \ref{hyp} below that, for $ \alpha>0 $, a very simple criterion on $ p$ ensures \eqref{hyp2}. With this criterion in hand, it is not too hard to check that the following linear operators satisfy  Hypothesis \ref{hyp1} :
 \begin{enumerate}
 \item The purely dispersive operators $ L:=\partial_x D_x^\alpha $ with $ \alpha>0$.
 \item The linear Intermediate Long Wave operator
 $ L:=\partial_x  D_x \coth(D_x) $. Note that here $ \alpha=1$.
 \item Some perturbations of the Benjamin-Ono equation as the Smith operator $ L:=\partial_x (D_x^2+1)^{1/2}$ (see \cite{smith}). Here again $ \alpha=1$.
  \end{enumerate}
 \end{remark}
\begin{theorem}\label{theo1}
Let $ \K=\R $ or $ \T $, $ L_{\alpha+1} $ satisfying Hypothesis \ref{hyp1} with $ 1\le \alpha\le 2 $ and let $ s\ge 1-\frac{\alpha}{2} $ with $ (s,\alpha) \neq (\frac{1}{2},1)$. Then the Cauchy problem associated
 with \eqref{bo} is unconditionally locally well-posed in $ H^s(\K ) $ with a maximal time of existence $ T\gtrsim (1+\|u_0\|_{H^{1-\frac{\alpha}{2}}})^{-\frac{2(\alpha+1)}{2\alpha-1}}$ .
\end{theorem}
 \begin{remark}\label{remark2}
 In the regular case (Cauchy problem in $H^s$ with $ s>1/2 $), we actually need \eqref{hyp2} only for $ |\xi_1|\wedge |\xi_2|\gg 1 $.
  \end{remark}

 \begin{remark}
 Our method also work in the case $ \alpha>2 $. In this case we  get the unconditional well-posedness in $ H^s(\K) $ for $ s\ge 0$
 .
  \end{remark}
\begin{remark}
In the appendix we indicate the small modifications that enable to obtain  the local-well posedness in the limit case $ (s,\alpha)=(1/2,1) $. However, in this  limit case we are not able to prove the unconditional uniqueness in $ L^\infty_T H^{1/2} $.
  \end{remark}
 \begin{remark}
For $ L_{\alpha+1}:=\partial_x^3 $, we recover  the unconditional LWP result in $  L^2(\T)  $ obtained in \cite{bit} for the KdV equation. However, our Lipschitz bound on the solution-map holds  at the level $ H^{-1/2} $ whereas in \cite{bit} it holds at the level $ H^{-1} $.
 Note that the $ L^2(\R) $ case was treated in \cite{zhou}.
  \end{remark}

Let us assume now that the symbol $p_{\alpha+1}$ satisfies moreover
\begin{equation}\label{hyp11}
 |p_{\alpha+1}(\xi)|\lesssim |\xi| \textrm{ for } |\xi|\le 1 \textrm{ and } |p_{\alpha+1}(\xi)| \sim |\xi|^{\alpha+1} \textrm{ for } |\xi|\ge 1.
\end{equation}
Then it is not too hard to check that equation (\ref{bo}) enjoys the following conservation laws:
\begin{align*}
 & \frac{d}{dt} \int_\K u^2 dx = 0,\\
 & \frac{d}{dt} \int_\K ( |\Lambda^{\alpha/2}u |^2 + \frac 13 u^3)dx = 0,
\end{align*}
where $\Lambda^{\alpha/2}$ is the space Fourier multiplier defined by
$$
\widehat{\Lambda^{\alpha/2} v}(\xi) = \left|\frac{p_{\alpha+1}(\xi)}{\xi}\right|^{1/2} \widehat{v}(\xi).
$$
Combined with the embedding $H^{\alpha/2} \hookrightarrow L^3$, we get an a priori bound of the $H^{\alpha/2}$-norm of the solution. We may then iterate Theorem \ref{theo1} and obtain the following corollary.
\begin{corollary}
 Let $ \K=\R $ or $ \T $, $ L_{\alpha+1} $ satisfying Hypothesis \ref{hyp1} and (\ref{hyp11}) with $ 1< \alpha\le 2 $ . Then the Cauchy problem associated
 with \eqref{bo} is unconditionally globally well-posed in $ H^{\alpha/2}(\K ) $  .
\end{corollary}

\begin{remark}
  The linear operators given in Remark \ref{remark-ex} also satisfy assumption (\ref{hyp11}).
\end{remark}

 It is well-known that gauge transform often do not well behave with respect to perturbation of the equation. On the other hand  it is well-known that,  taking into account some damping or dissipative effects, dissipative versions of \eqref{bo} can be derived (see for instance \cite{OS}, \cite{ER}).
 One quite direct application of the fact that we do not need a gauge transform to solve \eqref{bo}, is that we can easily treat  the dissipative limit  of  dissipative versions of \eqref{bo}. Such dissipative limit was for example studied for the Benjamin-Ono equation on the real line in
  \cite{guo} and \cite{Mo2}.

  Let us introduce the following dissipative version
 \begin{equation}\label{bob}
 u_t + L_{\alpha+1} u  +\varepsilon A_{\beta} u + u u_x =0
 \end{equation}
 where  $ \varepsilon>0 $ is a small parameter, $ \beta\ge 0 $ and $ A_{\beta} $ is a linear operator satisfying the following hypothesis :
 \begin{hypothesis}	\label{hypdis}
We assume that $ A_{\beta} $ is  the Fourier multiplier operator by $ q_{\beta} $ where  $ q_{\beta} $ is a real-valued  even function, bounded on bounded intervals, satisfying :
For all $ 0<\lambda\ll 1 $ and $ \xi \gg 1 $,
$$
\lambda^\beta q_\beta(\lambda^{-1} \xi) \sim |\xi|^\beta\, .
$$
 \end{hypothesis}
 \begin{remark}
 The homogeneous  operators $ D_x^\beta $ and the non homogeneous operators $J_x^\beta $ satisfy Hypothesis \ref{hypdis}.
 \end{remark}
\begin{theorem}\label{theo2}
 Let $ \K=\R $ or $ \T $, $ 1\le \alpha\le 2 $, $0\le \beta \le 1+\alpha $ and $ s\ge 1-\frac{\alpha}{2} $.
\begin{enumerate}
\item  Then the Cauchy problem  associated with \eqref{bob} is locally well-posed in $ H^s(\K) $.
\item For  $ u_0\in H^s(\K ) $,  let $ u $ be  solution
 to \eqref{bo} emanating from $ u_0$. We call $ T\gtrsim (1+\|u_0\|_{H^{1-\frac{\alpha}{2}}})^{-\frac{2(\alpha+1)}{2\alpha-1}}$ the maximal time of existence of $ u $ in
 $ H^s $. Then for $ \varepsilon>0 $ small enough, the maximal time of existence $ T_\varepsilon $ of
  the solution $ u_\varepsilon$ to
\eqref{bob} emanating from $ u_0$ satisfies $ T_\varepsilon\ge T $. Moreover, $ u_\varepsilon\to u $ in $C([0,T[; H^s)$ as $ \varepsilon\to 0 $ .
\end{enumerate}
\end{theorem}
\begin{remark}
The constraint $ \beta \le 1+\alpha $ is clearly an artefact of the method we used. We think that it could  be dropped by replacing, in some estimates, the dispersive Bourgain's spaces by dispersive-dissipative Bourgain's spaces that were first introduced in \cite{MR}. But since the dissipative operators involved in wave motions are commonly of order less or equal to $ 2$ we do not pursue this issue.
\end{remark}

The rest of the paper is organized as follows: in Section \ref{section2}, we introduce the notations, define the function spaces and state some preliminary lemmas. In Section \ref{section3} we develop our method in the simplest case $s>1/2$, while the non regular case is treated in Section \ref{section4}. Section \ref{section5} is devoted to the proof of Theorem \ref{theo2}. We conclude the paper with an appendix explaining how to deal with the special case $(s,\alpha)=(1/2,1)$.

\section{Notations, function spaces and preliminary lemmas}\label{section2}
\subsection{Notation}\label{sect21}
For any positive numbers $a$ and $b$, the notation $a \lesssim b$
means that there exists a positive constant $c$ such that $a \le c
b$. We also denote $a \sim b$ when $a \lesssim b$ and $b \lesssim
a$. Moreover, if $\alpha \in \mathbb R$, $\alpha_+$, respectively
$\alpha_-$, will denote a number slightly greater, respectively
lesser, than $\alpha$.

For $u=u(x,t) \in \mathcal{S}(\mathbb R^2)$,
$\mathcal{F}u=\widehat{u}$ will denote its space-time Fourier
transform, whereas $\mathcal{F}_xu=(u)^{\wedge_x}$, respectively
$\mathcal{F}_tu=(u)^{\wedge_t}$, will denote its Fourier transform
in space, respectively in time. For $s \in \mathbb R$, we define the
Bessel and Riesz potentials of order $-s$, $J^s_x$ and $D_x^s$, by
\begin{displaymath}
J^s_xu=\mathcal{F}^{-1}_x\big((1+|\xi|^2)^{\frac{s}{2}}
\mathcal{F}_xu\big) \quad \text{and} \quad
D^s_xu=\mathcal{F}^{-1}_x\big(|\xi|^s \mathcal{F}_xu\big).
\end{displaymath}

Throughout the paper, we fix a smooth cutoff function $\eta$ such that
\begin{displaymath}
\eta \in C_0^{\infty}(\mathbb R), \quad 0 \le \eta \le 1, \quad
\eta_{|_{[-1,1]}}=1 \quad \mbox{and} \quad  \mbox{supp}(\eta)
\subset [-2,2].
\end{displaymath}
We set  $ \phi(\xi):=\eta(\xi)-\eta(2\xi) $. For $l \in \mathbb Z$, we define
\begin{displaymath}
\phi_{2^l}(\xi):=\phi(2^{-l}\xi), \end{displaymath} and, for $ l\in \N^* $,
\begin{displaymath}
\psi_{2^{l}}(\xi,\tau)=\phi_{2^{l}}(\tau-p_{\alpha+1}(\xi)),
\end{displaymath}
where $ ip_{\alpha+1} $ is the Fourier symbol of $ L_{\alpha+1} $.
By convention, we also denote
\begin{displaymath}
\psi_{0}(\xi,\tau):=\eta(2(\tau-p_{\alpha+1}(\xi))).
\end{displaymath}
Any summations over capitalized variables such as $N, \, L$, $K$ or
$M$ are presumed to be dyadic with $N, \, L$, $K$ or $M > 0$,
\textit{i.e.}, these variables range over numbers of the form $\{2^k
: k \in \mathbb Z \}$. Then, we have that
\begin{displaymath}
\sum_{N}\phi_N(\xi)=1, \quad \mbox{supp} \, (\phi_N) \subset
\{\frac{N}{2}\le |\xi| \le 2N\}, \ N \ge 1, \quad \text{and} \quad
\mbox{supp} \, (\phi_0) \subset \{|\xi| \le 1\}.
\end{displaymath}
Let us define the Littlewood-Paley multipliers by
\begin{displaymath}
P_Nu=\mathcal{F}^{-1}_x\big(\phi_N\mathcal{F}_xu\big), \quad
Q_Lu=\mathcal{F}^{-1}\big(\psi_L\mathcal{F}u\big),
\end{displaymath}
 $P_{\ge N}:=\sum_{K \ge N} P_{K}$,  $P_{\le N}:=\sum_{K \le N} P_{K}$, $Q_{\ge L}:=\sum_{K \ge L} Q_{K}$ and   $Q_{\le L}:=\sum_{K \le L} Q_{K}$. For brevity we also write $u_N=P_Nu$, $u_{\le N}=P_{\le N}u$, ...

Let $\chi$ be a (possibly complex-valued) bounded function on $\R^2$ and define the pseudo-product operator $\Pi=\Pi_\chi$ on $\S(\R)^2$ by
$$
\F(\Pi(f,g))(\xi) = \int_{\R} \widehat{f}(\xi_1)\widehat{g}(\xi-\xi_1)\chi(\xi,\xi_1) d\xi_1.
$$
Throughout the paper, we write $\Pi=\Pi_\chi$ where $\chi$ may be different at each occurrence of $\Pi$.
This bilinear operator behaves like a product in the sense that it satisfies the following properties
$$\Pi(f,g) = fg \textrm{ if } \chi\equiv 1,$$
\begin{equation}\label{PiSym}
\int_\R \Pi(f,g)h = \int_\R f\Pi(g,h) = \int_\R \Pi(f,h)g
\end{equation}
for any $f,g,h\in \S(\R)$.
Moreover, we easily check from Bernstein inequality that if $f_i\in L^2(\R)$ has a Fourier transform localized in an annulus $\{|\xi|\sim N_i\}$, $i=1,2,3$, then
\begin{equation}\label{estPi}
\left|\int_\R \Pi(f_1,f_2)f_3\right|\lesssim N_{min}^{1/2}\prod_{i=1}^3\|f_i\|_{L^2},
\end{equation}
where the implicit constant only depends on $ \|\chi\|_{L^\infty(\R^2)} $ and $N_{min} = \min\{N_1,N_2,N_3\}$. With this notation in hand, we will be able to systematically estimate terms of the form
$$
\int_\R P_N(u^2)\partial_x P_N u
$$
to put the derivative on the lowest frequency factor.
\subsection{Function spaces} \label{spaces}
For $1 \le p \le \infty$, $L^p(\mathbb R)$ is the usual Lebesgue
space with the norm $\|\cdot\|_{L^p}$, and for $s \in \mathbb R$ ,
the real-valued Sobolev spaces $H^s(\mathbb R)$  denote the spaces of all real-valued functions with the usual
norms
\begin{displaymath}
\|\phi\|_{H^s}=\|J^s_x\phi\|_{L^2} \; .
\end{displaymath}
 If $B$
is one of the spaces defined above, $1 \le p \le \infty$, we will
define the space-time spaces $L^p_ t B$ and $ \widetilde{L^p_t} B $
 equipped with  the norms
\begin{displaymath}
\|f\|_{L^p_ t B} =\Big(\int_{\R}\|f(\cdot,t)\|_{B}^pdt\Big)^{\frac1p} ,
\end{displaymath}
with obvious modifications for $ p=\infty $,  and
\begin{displaymath}
\|f\|_{\widetilde{L^p_ t }B} =\Big(\sum_{N>0}
 \| P_N f \|_{L^p_ t B}^2\Big)^{\frac12} \; .
\end{displaymath}
For $s$, $b \in \mathbb R$, we introduce the Bourgain spaces
$X^{s,b}$ related to the linear part of \eqref{bo} as
the completion of the Schwartz space $\mathcal{S}(\mathbb R^2)$
under the norm
\begin{equation} \label{Bourgain}
\|v\|_{X^{s,b}} := \left(
\int_{\mathbb{R}^2}\langle\tau-p_{\alpha+1} (\xi)\rangle^{2b}\langle \xi\rangle^{2s}|\widehat{v}(\xi, \tau)|^2
d\xi d\tau \right)^{\frac12},
\end{equation}
where $\langle x\rangle:=1+|x|$ and $ i p_{\alpha+1} $ is Fourier symbol of $ L_{\alpha+1}$.
 Recall that $$ \|v\|_{X^{s,b}}= \| U_\alpha(-t) v \|_{H^{s,b}_{t,x}} $$ where $ U_\alpha(t)=\exp(t L_{\alpha+1}) $ is the generator of the free evolution associated with
  \eqref{bo}.\\
Finally, we will use restriction in time versions of these spaces.
Let $T>0$ be a positive time and $Y$ be a normed space of space-time functions. The restriction space $ Y_T $ will
 be the space of functions $v: \mathbb R \times
]0,T[\rightarrow \mathbb R$ satisfying
\begin{displaymath}
\|v\|_{Y_{T}}:=\inf \{\|\tilde{v}\|_{Y} \ | \ \tilde{v}: \mathbb R
\times \mathbb R \rightarrow \mathbb R, \ \tilde{v}|_{\mathbb R
\times ]0,T[} = v\}<\infty \; .
\end{displaymath}
\subsection{Preliminary lemmas}
\begin{lemma}\label{hyp}
Let $ p :\, \R \to \R $ be an odd function belonging to  $ C^1(\R) \cap C^2(\R^*) $ such that for all $ |\xi|\gg 1 $,
\begin{equation} \label{hyp5}
|p'(\xi)|\sim |\xi|^\alpha \quad \mbox{ and } \quad |p''(\xi)|\sim |\xi|^{\alpha-1} \, ,
\end{equation}
for some $ \alpha>0$.
Then  the Fourier multiplier  $L_{\alpha+1}$ by $ i \, p $ satisfies Hypothesis \ref{hyp1}.
\end{lemma}
\begin{proof}
By symmetry we can assume $ |\xi_2|\le |\xi_1| $.  
We separate different cases:\\
{\bf 1.} $|\xi_2|\ll |\xi_1| $. Since $|\xi_1|\gg 1$, we can assume that (\ref{hyp5}) holds for any $\xi\ge |\xi_1|$ and thus there exists $ \theta\in [\xi_1,\xi_1+\xi_2] $ such that
\begin{align*}
 \lambda^{\alpha+1} \Bigr| p(\lambda^{-1}(\xi_1+\xi_2))-p(\lambda^{-1} \xi_1) \Bigl|
& = \lambda^\alpha |\xi_2||p'(\lambda^{-1} \theta)|  \\
&\sim \lambda^\alpha |\xi_2| |\lambda^{-1}\theta|^\alpha \\
&\sim  |\xi_2||\xi_1|^\alpha
\end{align*}
 for all $ 0\le \lambda\ll 1 $.
On the other hand, if $ \lambda^{-1} |\xi_2| \le |\xi_1| $ then
$$
\lambda^{\alpha+1}|p(\lambda^{-1} \xi_2) |\le  \lambda^\alpha |\xi_2| \max_{\xi\in[0,|\xi_1|]} |p'(\xi)|\ll   |\xi_2||\xi_1|^\alpha
$$
and if $  \lambda^{-1} |\xi_2| \ge |\xi_1| $ then
\begin{align*}
\lambda^{\alpha+1}|p(\lambda^{-1} \xi_2) |& = \lambda^{\alpha+1} |p(\xi_1)+p(\lambda^{-1} \xi_2) -p( \xi_1)|   \\
&\le \lambda^{\alpha+1}\Bigr(|\xi_1| \max_{\xi\in[0,|\xi_1|]} |p'(\xi)|+ \lambda^{-1} |\xi_2| | \lambda^{-1} \xi_2 |^\alpha \Bigl) \\
&\le   |\xi_2|^{\alpha+1} +\lambda^{\alpha}|\xi_2| \max_{\xi\in[0,|\xi_1|]} |p'(\xi) |\ll   |\xi_2||\xi_1|^\alpha
\end{align*}
Gathering these two estimates leads for $ 0<\lambda \ll 1 $ to
$$
\lambda^{\alpha+1} |\Omega(\lambda^{-1}\xi_1,\lambda^{-1} \xi_2) | \sim |\xi_2||\xi_1|^\alpha \, .
$$
{\bf 2.} $|\xi_2|\gtrsim |\xi_1| $. In this case we can assume that (\ref{hyp5}) holds for any $ \xi\ge |\xi_2| $. \\
{\bf 2.1.} $ \xi_1.\xi_2\ge 0$. Then we have $0< \xi_2\le \xi_1 <\xi_1+\xi_2$. We notice that
\begin{align*}
\lambda^{\alpha+1} |\Omega(\lambda^{-1}\xi_1,\lambda^{-1} \xi_2) |  = &
 \lambda^{\alpha} \int_{\lambda \xi_2}^{\xi_2} (p'(\lambda^{-1}(\xi_1+\theta))-p'(\lambda^{-1} \theta)) \, d\theta \\
& + \lambda^{\alpha+1}\Bigr(p(\lambda^{-1}\xi_1+\xi_2)-p(\lambda^{-1}\xi_1)\Bigl)-\lambda^{\alpha+1}p(\xi_2) \, .
\end{align*}
with
$$ |p(\lambda^{-1}\xi_1+\xi_2)-p(\lambda^{-1}\xi_1)| \lesssim \xi_2 \lambda^{-\alpha} \xi_1^\alpha\ll \lambda^{-\alpha-1} \xi_2 \xi_1^\alpha
$$
and
\begin{align*}
p'(\lambda^{-1}(\xi_1+\theta))-p'(\lambda^{-1} \theta)
& =  \lambda^{-1}
\int_{0}^{\xi_1} p''(\lambda^{-1}(\theta+\mu)) \, d\mu.
\end{align*}
But for $ \xi\ge \xi_2 $, $p''$ does not change sign since $ |p''(\xi)|\sim |\xi|^{\alpha-1} $ and $ p''$ is continuous outside $ 0$.
 Therefore,
 \begin{align*}
 \lambda^{-1} \int_{0}^{\xi_1} p''(\lambda^{-1}(\theta+\mu)) \, d\mu
 & \sim \lambda^{-1} \int_{0}^{\xi_1} (\lambda^{-1}(\theta+\mu))^{\alpha-1} \, d\mu \\
  & \sim \lambda^{-\alpha}\Bigl( (\xi_1+\theta)^\alpha-\theta^\alpha\Bigr) \\
 &  \sim \lambda^{-\alpha}\xi_1^\alpha
 \end{align*}
 Gathering these estimates we obtain
 $$
 \lambda^{\alpha+1} |\Omega(\lambda^{-1}\xi_1,\lambda^{-1} \xi_2) |  \sim \xi_2 \xi_1^\alpha \, .
 $$
 {\bf 2.2.} $\xi_1.\xi_2 <0 $. We can assume that $ \xi_1>0$. Then we have $ 0<\xi_1+\xi_2<-\xi_2\le \xi_1$.
  For $\xi_1+\xi_2\ll -\xi_2 $, recalling that $ p $ is an odd function, we can argue exactly as in the case  {\bf 1.} but with $ \xi_1+\xi_2$ ,  $-\xi_2$ and  $\xi_1$ playing the role of respectively
  $ \xi_2$, $\xi_1$ and $ \xi_1+\xi_2$. Finally, for $ \xi_1+\xi_2\gtrsim  -\xi_2 $, we argue exactly as in the case  {\bf 2.1} with the same exchange of roles  than above.
\end{proof}
\begin{lemma}\label{resonance.lem}
Assume that  $ p_{\alpha+1} $ satisfies  \eqref{hyp2} with $ \lambda=1 $.
Let $L_1,L_2,L_3 >0 $, $0<N_1\le N_2\le N_3 $ be dyadic numbers and $ u,v,w\in {\mathcal S}'(\R^2) $.
Then
$$
\int_{\R^2} \Pi(Q_{L_1}P_{N_1} u , \, Q_{L_2} P_{N_2} v) \, Q_{L_3} P_{N_3} w =0
$$
whenever the following relation is not satisfied :
\begin{equation}\label{resonance}
\max(L_1,L_2,L_3)  \sim \max( N_1 N_2^{\alpha}, L_{med})\; ,
\end{equation}
where $ L_{med}= \max(\{L_1,L_2,L_3\}-\{L_{\max}\})$.
\end{lemma}
\begin{proof}
This is a direct consequence of the hypothesis  \eqref{hyp2}  on  the resonance function $ \Omega(\xi_1,\xi_2)$ since
$$
\Omega(\xi_1,\xi_2)=\sigma(\tau_1+\tau_2,\xi_1+\xi_2) - \sigma(\tau_1,\xi_1) -\sigma(\tau_2,\xi_2)
$$
with  $ \sigma(\tau,\xi):=\tau-p_{\alpha+1}(\xi) $.
\end{proof}

\begin{lemma}\label{continuiteQ}
Let $ L>0 $, $ 1\le p \le \infty $ and $ s\in \R $.
The operator  $ Q_{\le L} $ is bounded  in $ L^p_t H^s $ uniformly in $ L>0$.
\end{lemma}
\begin{proof}
Let $ R_{\le L} $ be the Fourier multiplier by $ \eta_L(\tau) $ where $ \eta_L $ is defined in Section \ref{sect21}.
The trick is to notice that $Q_{\le L} u = U_\alpha(t) (R_{\le L} U_\alpha(-t) u )$. Therefore, using the unitarity of $ U_\alpha(\cdot) $ in
  $ H^s(\R) $ we infer that
\begin{align*}
\|Q_{\le L} u \|_{L^p_t H^s} &= \|U_\alpha(t) (R_{\le L} U_\alpha(-t) u ) \|_{L^p_t H^s}
  = \|R_{\le L} U_\alpha(-t) u  \|_{L^p_t H^s}\\
& \lesssim  \| U_\alpha (-t) u  \|_{L^p_t H^s}
 = \|u  \|_{L^p_t H^s} \; .
\end{align*}
\end{proof}
For any $T>0$, we consider $1_T$ the characteristic function of $[0,T]$ and use the decomposition
\begin{equation}\label{ind-dec}
1_T = 1_{T,R}^{low}+1_{T,R}^{high},\quad \widehat{1_{T,R}^{low}}(\tau)=\eta(\tau/R)\widehat{1_T}(\tau)
\end{equation}
for some $R>0$.\\
The properties of this decomposition we will need are listed in the following lemmas.
\begin{lemma}\label{ihigh-lem} For any $ R>0 $ and $ T>0 $ it holds
\begin{equation}\label{high}
\|1_{T,R}^{high}\|_{L^1}\lesssim T\wedge R^{-1}.
\end{equation}
and 
\begin{equation}\label{low}
\|1_{T,R}^{low}\|_{L^\infty}\lesssim  1.
\end{equation}
\end{lemma}
\begin{proof} A direct computation provide
\begin{align*}\notag
\|1_{T,R}^{high}\|_{L^1} &= \int_\R \left|\int_\R (1_T(t)-1_T(t-s/R))\F^{-1}\eta(s)ds\right|dt\\
\notag &\le \int_\R\int_{([0,T]\setminus [s/R,T+s/R])\cup ([s/R,T+s/R]\setminus [0,T])}|\F^{-1}\eta(s)|dtds\\
\notag &\lesssim \int_\R (T\wedge |s|/R) |\F^{-1}\eta(s)|ds\\
\label{indT-h}&\lesssim T\wedge R^{-1}.
\end{align*}
Finally, the proof of \eqref{low} is obvious.
\end{proof}
\begin{lemma}\label{ilow-lem}
Let $u\in L^2(\R^2)$. Then for any $T>0$, $R>0$  and $ L \gg R $ it holds
$$
\|Q_L (1_{T,R}^{low}u)\|_{L^2}\lesssim \|Q_{\sim L} u\|_{L^2}
$$
\end{lemma}

\begin{proof}
By Plancherel we get
\begin{align*}
I_{L} &=\|Q_L(1_{T,R}^{low}u)\|_{L^2} \\
&=\|\varphi_L(\tau-\omega(\xi))\widehat{1_{T,R}^{low}}\ast_\tau \widehat{u}(\tau,\xi)\|_{L^2} \\
&= \left\|\sum_{L_1}\varphi_L(\tau-\omega(\xi))\int_\R \varphi_{L_1}(\tau'-\omega(\xi))\widehat{u}(\tau',\xi)\eta((\tau-\tau')/R)\frac{e^{-iT(\tau-\tau')}-1}{\tau-\tau'}d\tau'\right\|_{L^2}.
\end{align*}
In the region where $L_1\ll L$ or $L_1\gg L$, we have $|\tau-\tau'|\sim L\vee L_1\gg R$, thus $I_L$ vanishes.
On the other hand, for  $L\sim L_1$, we get
$$
I_L \lesssim \sum_{L\sim L_1} \|Q_L(1_{T,R}^{low}Q_{L_1}u)\|_{L^2}
\lesssim \|Q_{\sim L}u\|_{L^2}.
$$
\end{proof}

\section{Unconditional well-posedness in the regular case $ s>1/2$}\label{section3}
In this section we develop our method in the regular case $ s>1/2$. This will emphasize the simplicity of this approach to prove unconditional well-posedness for equation \eqref{bo} posed on $ \R $ or $\T $.

Let $ \lambda>0 $ and $ L^\lambda_{\alpha+1}$ be the Fourier multiplier by $ i  \lambda^{\alpha+1} p_{\alpha+1}(\lambda^{-1} \cdot) $. We  notice that if $ u $ is solution to \eqref{bo} on $ ]0,T[ $ then
$ u_\lambda(t,x)=\lambda^{\alpha} u(\lambda^{\alpha+1}t, \lambda x) $ is solution to  \eqref{bo} on $ ]0,\lambda^{-(\alpha+1)} T[ $ with $ L_{\alpha+1} $ replaced by $ L^\lambda_{\alpha+1}$. Therefore, up to this change of unknown and equation, we can always assume that the operator
 $ L_{\alpha+1} $  verifies \eqref{hyp2} with $0< \lambda \le 1$.

\subsection{A priori estimates}
For $ s\in \R $ we define the function space $ M^s $ as $M^s:= L^\infty_t H^s \cap X^{s-1,1} $, endowed with the natural norm
$$
\|u\|_{M^s}= \|u\|_{L^\infty_t H^s}+\|u\|_{X^{s-1,1}} \; .
$$
For $ u_0\in H^s(\R) $, $s>1/2$, we will construct a solution to \eqref{bo} in $ M^s_T $, whereas the difference of two solutions emanating from initial data belonging to $ H^s(\R) $ will take place in $ M^{s-1}_T $.
\begin{lemma}\label{lem1}
Let $ 0<T<2$, $ s>1/2 $  and $ u\in L^\infty_T H^s $  be a solution to \eqref{bo}.
Then $ u\in M^s_T$ and it holds
\begin{equation}\label{estXregular}
\|u\|_{M^s_T}  \lesssim \|u\|_{L^\infty_T H^s} + \|u\|_{L^\infty_T H^s} \|u\|_{L^\infty_T H^{\frac{1}{2}+}}\;.
\end{equation}
Moreover, for any couple $(v,w) \in   (L^\infty_T H^s)^2 $ of solutions to \eqref{bo}, it holds
\begin{equation}\label{estdiffXregular}
\|u-v\|_{M^{s-1}_T}  \lesssim \|u-v\|_{L^\infty_T H^{s-1}} +  \|u+v\|_{L^\infty_T H^{s}}  \|u-v\|_{L^\infty_T H^{s-1}} \; .
\end{equation}
\end{lemma}
\begin{proof}
We have to extend the function $ u$ from $ (0,T) $ to $ \R $. For this we follow \cite{MN} and introduce the extension operator $ \rho_T $ defined by
\begin{equation}\label{defrho}
\rho_T u(t):= \eta(t) u(T \mu(t/T))\; ,
\end{equation}
where $ \eta $ is the smooth cut-off function defined in Section \ref{sect21} and $ \mu(t)=\max(1-|t-1|,0) $. $ \rho_T $ is a bounded operator from
 $ X^{\theta,b}_T $ into $ X^{\theta,b} $ and from $ L^p(0,T;X) $ into $ L^p(\R;X) $ for any $ b\in ]-\infty,1]$, $ s\in \R $, $ p\in [1,\infty] $ and any Banach space
  $ X$. Moreover, these bounds are uniform for $ 0<T<1 $. \\
  By using this extension operator, it is clear that we only have to estimate the $ X^{s-1,1}_T $-norm of $ u $ to prove \eqref{estXregular}. But by the Duhamel formula of \eqref{bo} and  standard linear estimates in Bourgain's spaces, we have
  \begin{eqnarray*}
  \|u\|_{X^{s-1,1}_T} & \lesssim &  \|u_0\|_{H^{s-1}}+\| \partial_x(u^2) \|_{X^{s-1,0}_T} \\
  & \lesssim &  \|u_0\|_{H^{s-1}}+\|u^2 \|_{L^2_T H^s} \\
 &  \lesssim   &  \|u_0\|_{H^{s-1}}+ \|u\|_{L^\infty_T H^{\frac{1}{2}+}} \| u \|_{L^\infty_T H^s}  \; ,
 \end{eqnarray*}
by standard product estimates in Sobolev spaces.\\
In the same way we have
  \begin{eqnarray*}
  \|u-v\|_{X^{s-2,1}_T}
  & \lesssim &  \|u_0\|_{H^{s-2}}+\|(u+v)(u-v) \|_{L^2_T H^{s-1}} \\
 &  \lesssim   &  \|u_0\|_{H^{s-2}}
 +  \|u+v\|_{L^\infty_T H^{s}} \| u-v \|_{L^\infty_T H^{s-1}} \; ,
 \end{eqnarray*}
 which proves \eqref{estdiffXregular}
\end{proof}

\begin{lemma}\label{lemtriest}
Assume $u_i\in M^0$, $i=1,2,3$ are functions with spatial Fourier support in $\{|\xi|\sim N_i\}$ with $N_i>0$ dyadic satisfying $N_1\le N_2\le N_3$. For any $t>0$ we set
$$
I_t(u_1,u_2,u_3) = \int_0^t\int_\R \Pi(u_1,u_2)u_3.
$$
If $N_1\le 2^9$ it holds
\begin{equation}\label{estitlow}
|I_t(u_1,u_2,u_3)|\lesssim N_1^{1/2} \|u_1\|_{L^\infty_tL^2_x} \|u_2\|_{L^2_{tx}}\|u_3\|_{L^2_{tx}}.
\end{equation}
In the case $N_1>2^9$ it holds
\begin{align*}
|I_t(u_1,u_2,u_3)|\lesssim & N_1^{-1/2}N_3^{1-\alpha} \|u_1\|_{L^\infty_tL^2_x}(\|u_2\|_{L^2_{tx}} \|u_3\|_{X^{-1,1}} + \|u_2\|_{X^{-1,1}} \|u_3\|_{L^2_{tx}})\\
&+ N_1^{1/2}N_3^{-\alpha} \|u_1\|_{X^{-1,1}} \|u_2\|_{L^2_{tx}} \|u_3\|_{L^\infty_tL^2_x}\\
&+ N_1^{-1}N_3^{-1/8} \|u_1\|_{L^\infty_tL^2_x} \|u_2\|_{L^\infty_tL^2_x} \|u_3\|_{L^\infty_tL^2_x}.
\end{align*}
\end{lemma}
\begin{proof}
Estimate (\ref{estitlow}) easily follows from (\ref{estPi}) together with H\"older inequality, thus it suffices to estimate $|I_t|$ for $N_1>2^9$.
Note that $I_t$ vanishes unless $N_2\sim N_3$. Setting $R=N_1^{3/2}N_3^{1/8}$, we split $I_t$ as
\begin{align}
I_t(u_1,u_2,u_3) &= I_\infty(1_{t,R}^{high}u_1,u_2, u_3) + I_\infty(1_{t,R}^{low}u_1,u_2, u_3) \nonumber \\
& :=I_t^{high} + I_t^{low}, \label{decIt}
\end{align}
where $I_\infty(u,v,w) = \int_{\R^2}\Pi(u,v)w$. The contribution of $I_t^{high}$ is estimated thanks to Lemma \ref{ihigh-lem} as well as (\ref{estPi}) and H\"older inequality by
\begin{align}
I_t^{high} &\lesssim N_1^{1/2} \|1_{t,R}^{high}\|_{L^1} \|u_1\|_{L^\infty_tL^2_x} \|u_2\|_{L^\infty_tL^2_x}\|u_3\|_{L^\infty_tL^2_x} \label{estIthigh} \\
&\lesssim N_1^{-1}N_3^{-1/8} \|u_1\|_{L^\infty_tL^2_x} \|u_2\|_{L^\infty_tL^2_x}\|u_3\|_{L^\infty_tL^2_x} \nonumber
\end{align}
To evaluate the contribution $ I_t^{low} $ we use that, according to  Lemma \ref{resonance.lem}, we have the following decomposition:
\begin{align}
I_\infty(1_{t,R}^{low} \, u_1,u_2, u_3)  =  & I_\infty(Q_{>2^{-4}N_1 N_3^\alpha}( 1_{t,R}^{low}u_1),u_2,  u_3)\nonumber \\
& +I_\infty(Q_{\le 2^{-4}N_1 N_2^\alpha} ( 1_{t,R}^{low}u_1), Q_{>2^{-4}N_1 N_3^\alpha}  u_2, u_3 )\nonumber \\
& + I_\infty(Q_{\le 2^{-4}N_1 N_2^\alpha}  ( 1_{t,R}^{low}u_1), Q_{\le 2^{-4}N_1 N_3^\alpha}  u_2, Q_{\sim N_1 N_3^\alpha}u_3)
\label{AA}\; .
\end{align}
It is worth noticing that since $ N_1\ge 2^9 $,  $R \ll 2^{-4} N_1 N_3^\alpha $.
Therefore the  contribution  $  I_t^{1,low} $ of the first term of the above RHS  to $I_t^{low} $ is easily estimated thanks to Lemma \ref{ilow-lem} by
\begin{align}
 I_t^{1,low}    \lesssim &  N_1^{1/2}(N_1N_3^\alpha)^{-1} \|u_1 \|_{X^{0,1} } \|u_2 \|_{L^2_{tx}} \|u_3 \|_{L^\infty_t L^2_x} \nonumber \\
\lesssim & N_1^{1/2}N_3^{-\alpha} \|u_1 \|_{X^{-1,1} } \|u_2 \|_{L^2_{tx}} \|u_3 \|_{L^\infty_t L^2_x}\; .
\end{align}
Thanks to Lemmas \ref{continuiteQ} and \ref{ilow-lem}, the contribution $ I_t^{2,low} $ of the second term can be handle in the following way
  \begin{align}
 I_t^{2,low} \lesssim &   N_1^{1/2}(N_1 N_3^\alpha)^{-1}
 \|u_1 \|_{L^\infty_t L^2_x } \|u_2 \|_{X^{0,1}} \|u_3 \|_{L^\infty_tL^2_x} \nonumber \\
   \lesssim & N_1^{-1/2}N_3^{1-\alpha} \|u_1 \|_{L^\infty_t L^2_x } \|u_2 \|_{X^{-1,1}} \|u_3 \|_{L^2_{tx}} \; .
\end{align}
Finally the contribution of the third term is estimated in the same way.
\end{proof}
\begin{remark}
From (\ref{PiSym}) we see that estimates in Lemma \ref{lemtriest} also hold for any other rearrangements of $N_1, N_2, N_3$.
\end{remark}
We are now in position to derive our  ``improved" energy estimate on smooth solutions to \eqref{bo}.
\begin{proposition}\label{prou}
Let $ 0<T<2$ and $ u\in L^\infty_T H^s $ with $ s>1/2 $ be a solution to \eqref{bo}. Then it holds
\begin{equation}\label{estHsregular}
\|u\|_{L^\infty_TH^s}  \lesssim \|u_0\|_{H^s}  +  (1+\|u\|_{L^\infty_T  H^{\frac{1}{2}+}}^2) \|u\|_{L^\infty_T H^{\frac{1}{2}+}} \|u\|_{L^\infty_T H^s} \;.
\end{equation}
\end{proposition}
\begin{proof}
Applying the operator $ P_N $ with $ N>0  $ dyadic  to equation \eqref{bo}, taking the $H^s$ scalar product with $ P_N u $ and integrating
 on $ ]0,t[ $  we obtain
\begin{equation}
\|P_N u\|^2_{L^\infty_T H^s} \lesssim \|P_N u_0\|_{H^s}^2 + \sup_{t\in ]0,T[}  \langle N\rangle^{2s} \Bigl| \int_0^t\int_{\R}  P_N(u^2) \partial_xP_N u \Bigr|
\end{equation}
Thus it remains to estimate
\begin{equation}\label{defI}
 I :=\sum_{N>0}  \langle N\rangle^{2s}\sup_{t\in ]0,T[}  \Bigl| \int_0^t\int_{\R}  P_N(u^2) \partial_xP_N u \Bigr| \; .
\end{equation}
We take an extension  $\tilde{u} $ of $ u $ supported in time in $ ]-2,2[ $  such that $ \|\tilde{u}\|_{M^s} \lesssim \|u\|_{M^s_T} $ . To simplify the notation we drop the tilde in the sequel.\\
By localization considerations, we get
\begin{equation}\label{decdy}
P_N(u^2) = P_N(u_{\gtrsim N}u_{\gtrsim N}) + 2P_N(u_{\ll N}u).
\end{equation}
Moreover, using a Taylor expansion of $\phi_N$, we easily get
\begin{equation}\label{taylorphi}
P_N(u_{\ll N}u) = u_{\ll N} P_Nu + N^{-1}\Pi(\partial_x u_{\ll N}, u),
\end{equation}
where $\Pi=\Pi_\chi$ with $\chi(\xi,\xi_1)=-i\int_0^1\phi'(N^{-1}(\xi-\theta\xi_1))d\theta\in L^\infty$. Inserting (\ref{decdy})-(\ref{taylorphi}) into (\ref{defI}) and integrating by parts we deduce
\begin{align*}
I \lesssim & \sum_{N>0} \sum_{N_1\ll N} N_1 \cro{N}^{2(s-1)} \sup_{t\in ]0,T[} \left| \int_0^t\int_\R \Pi_{\chi_1}(u_{N_1}, u_N) u_N\right|\\
& + \sum_{N>0} \sum_{N_1\ll N} N_1 \cro{N}^{2(s-1)} \sup_{t\in ]0,T[} \left| \int_0^t\int_\R \Pi_{\chi_2}(u_{N_1}, u_{\sim N}) u_N\right|\\
& + \sum_{N>0} \sum_{N_1\gtrsim N} N \cro{N}^{2(s-1)} \sup_{t\in ]0,T[} \left| \int_0^t\int_\R \Pi_{\chi_3}(u_{N_1}, u_{\sim N_1}) u_N\right|
\end{align*}
where $\chi_i$, $1\le i\le 3$ are bounded uniformly in $N,N_1$ and defined by
\begin{align}
\chi_1(\xi,\xi_1) &= \frac{\xi_1}{N_1}1_{\supp \phi_{N_1}}(\xi_1), \label{chi1} \\
\chi_2(\xi,\xi_1) &= \chi(\xi,\xi_1) \frac{\xi_1}{N_1} \frac{\xi}{N} \frac{1_{\supp \phi_N}(\xi)1_{\supp \phi_{N_1}}(\xi_1)}{\phi_{\sim N}(\xi-\xi_1)} \label{chi2} \\
\chi_3(\xi,\xi_1) &= \frac{\xi}{N}\phi_N(\xi) \label{chi3}.
\end{align}
Recalling now the definition of $I_t$ (see Lemma \ref{lemtriest}), it follows that
\begin{equation}\label{estI1}
 I \lesssim  \sum_{N>0}\sum_{N_1\gtrsim  N}  N\langle N_1\rangle^{2s}\sup_{t\in ]0,T[} |I_t(u_N, u_{\sim N_1}, u_{N_1})|.
  \end{equation}
The contribution of the sum over $ N \le 2^9 $ is easily estimated thanks to (\ref{estitlow}) and Cauchy-Schwarz by
\begin{align}
\sum_{N\le 2^9 }\sum_{N_1\gtrsim N} & N \langle N_1\rangle^{2s}\,\|u_N \|_{L^\infty_t L^2_x}
\|u_{N_1}\|_{L^2_t L^2_x}^2 \nonumber \\
& \lesssim \|u\|_{L^\infty_t L^2_x} \|u\|_{L^\infty_t H^s}^2 \label{estI2}
\end{align}
Finally the contribution of the sum over $N>2^9$ is controlled with the second part of Lemma \ref{lemtriest} by
\begin{align}
\sum_{N>2^9} \sum_{N_1\gtrsim N} NN_1^{2s} &\Big[ N^{-1/2}N_1^{1-\alpha} \|u_N\|_{L^\infty_tL^2_x} \|u_{N_1}\|_{L^2_{tx}} \|u_{N_1}\|_{X^{-1,1}} \nonumber \\
& \quad + N^{1/2}N_1^{-\alpha} \|u_N\|_{X^{-1,1}} \|u_{N_1}\|_{L^\infty_tL^2_x}^2 \nonumber\\
& \quad + N^{-1}N_1^{-1/8} \|u_N\|_{L^\infty_tL^2_x} \|u_{N_1}\|_{L^\infty_tL^2_x}^2 \Big] \nonumber \\
&\lesssim \|u\|_{M^{\frac{1}{2}+}_T}  \|u\|_{M^{s}_T} \|u\|_{L^\infty_T H^s}.
\end{align}
Gathering all the above estimates  leads to
\begin{eqnarray}
\label{est1}
\|u\|_{L^\infty_T H^s}^2 \lesssim \|u_0\|_{H^s}^2 + \|u\|_{M^{\frac{1}{2}+}_T}  \|u\|_{M^{s}_T} \|u\|_{L^\infty_T H^s}
\end{eqnarray}
which, together with \eqref{estXregular} completes the proof of the proposition.
\end{proof}
Let us now establish an a priori estimate at the regularity level $ s-1$ on the difference of two solutions.
\begin{proposition} \label{prodif}
Let $ 0<T<2$ and $ u,v \in L^\infty_T H^s $ with $ s>1/2 $ be two solutions to \eqref{bo}. Then it holds
\begin{equation}\label{estdiffHsregular}
\|u-v\|_{L^\infty_TH^{s-1}}  \lesssim \|u_0-v_0\|_{H^{s-1}}  +\|u+v\|_{M^s_T}
\|u-v\|_{M^{s-1}_T} \;.
\end{equation}
\end{proposition}
\begin{proof}
 The difference $w=u-v$ satisfies
\begin{equation}\label{eq-diff}
w_t+D^\alpha w_x = \partial_x(zw),
\end{equation}
where $z=u+v$. Proceeding  as in the proof of the preceding proposition, we infer that for $ N>0$,
\begin{equation}
\|P_N w\|^2_{L^\infty_T H^{s-1}} \lesssim \|P_N w_0\|_{H^{s-1}}^2 + \sup_{t\in ]0,T[}  \langle N\rangle^{2(s-1)} \Bigl| \int_0^t\int_{\R}  P_N(zw) \partial_xP_N w \Bigr|
\end{equation}
We take extensions  $\tilde{z} $ and $ \tilde{w} $  of $ z $ and $ w$ supported in time in $]-2,2[ $   such that $ \|\tilde{z}\|_{M^s} \lesssim \|u\|_{M^s_T} $
 and  $ \|\tilde{w}\|_{M^{s-1}} \lesssim \|u\|_{M^{s-1}_T} $. To simplify the notation we drop the tilde in the sequel.\\
Setting
\begin{equation}\label{defJ}
 J :=\sum_{N>0} \langle N\rangle^{2(s-1)}  \sup_{t\in ]0,T[}\Bigl| \int_0^t\int_{\R}  P_N(zw) \partial_x P_N w \Bigr|
\end{equation}
it follows from (\ref{taylorphi}) and classical dyadic decomposition   that for all $N>0$,
\begin{align}
P_N(zw) &= P_N(z_{\ll N}w) + P_N(z_{\sim N}w_{\lesssim N}) + \sum_{N_1\gg N} P_N(z_{N_1} w_{\sim N_1}) \nonumber \\
&= z_{\ll N}w_N + N^{-1}\Pi_\chi(\partial_xz_{\ll N}, w) + P_N(z_{\sim N}w_{\lesssim N}) + \sum_{N_1\gg N} P_N(z_{N_1} w_{\sim N_1}).\label{dec1}
\end{align}
Inserting this into (\ref{defJ}) and integrating by parts we infer
\begin{align*}
 J &\lesssim \sum_{N>0} \sum_{N_1\ll N} N_1 \cro{N}^{2(s-1)} \sup_{t\in ]0,T[} \left| \int_0^t\int_\R \Pi_{\chi_1}(z_{N_1}, w_N) w_N \right|\\
 &\quad + \sum_{N>0} \sum_{N_1\ll N} N_1 \cro{N}^{2(s-1)} \sup_{t\in ]0,T[} \left| \int_0^t\int_\R \Pi_{\chi_2}(z_{N_1}, w_{\sim N}) w_N \right|\\
 &\quad + \sum_{N>0} \sum_{N_1\lesssim N} N \cro{N}^{2(s-1)} \sup_{t\in ]0,T[} \left| \int_0^t\int_\R \Pi_{\chi_3}(z_{\sim N}, w_{N_1}) w_N \right|\\
 &\quad + \sum_{N>0} \sum_{N_1\gg N} N \cro{N}^{2(s-1)} \sup_{t\in ]0,T[} \left| \int_0^t\int_\R \Pi_{\chi_3}(z_{N_1}, w_{\sim N_1}) w_N \right|
\end{align*}
where $\chi_i$, $1\le i\le 3$ are defined in (\ref{chi1})-(\ref{chi2})-(\ref{chi3}). Therefore it suffices to estimate
\begin{align}
J &\lesssim \sum_{N>0}\sum_{N_1\gtrsim N} N\langle N_1\rangle^{2(s-1)} \sup_{t\in ]0,T[} |I_t(z_N, w_{\sim N_1}, w_{N_1})| \nonumber \\
&\quad + \sum_{N>0}\sum_{N_1\gtrsim N} N_1\langle N_1\rangle^{2(s-1)} \sup_{t\in ]0,T[} |I_t(z_{\sim N_1}, w_N, w_{N_1})| \nonumber \\
&\quad + \sum_{N>0}\sum_{N_1\gtrsim N} N\langle N\rangle^{2(s-1)} \sup_{t\in ]0,T[} |I_t(z_{N_1}, w_{N_1}, w_N)| \nonumber \\
&:= J_1+J_2+J_3 \label{J1}.
\end{align}
  The contribution of the sum over $ N \le 2^9 $ in \eqref{J1} is easily estimated thanks to (\ref{estitlow}) by
\begin{align}
\sum_{N\le 2^9 }\sum_{N_1\gtrsim N} & N^{1/2} \Bigl(
N \,\|z_N\|_{L^\infty_tL^2_x} \|w_{N_1} \|_{L^2_tH^{s-1}}^2 \nonumber \\
&\quad +N_1 \langle N_1\rangle^{-1} \,\|z_{N_1} \|_{L^2_t H^s}
\|w_N\|_{L^\infty_t L^2_x} \|w_{N_1} \|_{L^2_t H^{s-1}} \nonumber \\
&\quad + N \langle N_1\rangle^{1-2s}  \,\|z_{N_1} \|_{L^2_t H^s}
\|w_{N_1}\|_{L^2_t H^{s-1}} \|w_N \|_{L^\infty_t L^2_x}\Bigr) \nonumber \\
& \lesssim \|z\|_{L^\infty_t L^2_x}  \|w\|_{L^\infty_t H^{s-1}}^2 +  \|w\|_{L^\infty_t H^{-\frac{1}{2}}_x}\|z\|_{L^\infty_t H^s} \|w\|_{L^\infty_t H^{s-1}}
\end{align}
For the contribution of the sum over $N> 2^9$, it is worth noticing that since $s>1/2$, the term $J_3$ is controlled by $J_2$. The contribution of $J_1$ is estimated thanks to Lemma \ref{lemtriest} by
\begin{align}
\sum_{N>2^9}\sum_{N_1\gtrsim N} N N_1^{2(s-1)} & \Big[ N^{-1/2}N_1^{1-\alpha} \|z_N\|_{L^\infty_tL^2_x} \|w_{N_1}\|_{L^2_{tx}} \|w_{N_1}\|_{X^{-1,1}} \nonumber \\
&\quad + N^{1/2}N_1^{-\alpha} \|z_N\|_{X^{-1,1}} \|w_{N_1}\|_{L^\infty_tL^2_x}^2 \nonumber \\
&\quad + N^{-1}N_1^{-1/8} \|z_N\|_{L^\infty_tL^2_x} \|w_{N_1}\|_{L^\infty_tL^2_x}^2 \Big] \nonumber \\
&\lesssim \|z\|_{M^{1/2+}} \|w\|_{M^{s-1}} \|w\|_{L^\infty_tH^{s-1}}.
\end{align}
Finally, we bound in the same way $J_2$ by
\begin{align}
\sum_{N>2^9}\sum_{N_1\gtrsim N} N_1^{2s-1} & \Big[ N^{-1/2}N_1^{1-\alpha} \|w_N\|_{L^\infty_tL^2_x} (\|z_{N_1}\|_{L^2_{tx}} \|w_{N_1}\|_{X^{-1,1}} + \|z_{N_1}\|_{X^{-1,1}} \|w_{N_1}\|_{L^2_{tx}}) \nonumber \\
&+ N^{1/2}N_1^{-\alpha} \|w_N\|_{X^{-1,1}} \|z_{N_1}\|_{L^\infty_tL^2_x} \|w_{N_1}\|_{L^\infty_tL^2_x} \nonumber\\
&+ N^{-1}N_1^{-1/8} \|w_N\|_{L^\infty_tL^2_x} \|z_{N_1}\|_{L^\infty_tL^2_x} \|w_{N_1}\|_{L^\infty_tL^2_x} \Big] \nonumber\\
&\lesssim \|z\|_{M^s} \|w\|_{M^{-1/2+}} \|w\|_{L^\infty_tH^{s-1}} + \|z\|_{M^s} \|w\|_{M^{s-1}} \|w\|_{L^\infty_tH^{-1/2+}}. \label{J4}
\end{align}
Gathering estimates \eqref{J1}-\eqref{J4}  we obtain
\begin{equation}
\label{est2}
J \lesssim  (\|z\|_{M^{\frac{1}{2}+}_T} \|w\|_{M^{s-1}} + \|z\|_{M^{s}_T} \|w\|_{M^{-\frac{1}{2}+}})\|w\|_{L^\infty_T H^{s-1}} + \|z\|_{M^{s}_T} \|w\|_{M^{s-1}}\|w\|_{L^\infty_T H^{-\frac{1}{2}+}}
\end{equation}
which  leads to \eqref{estdiffHsregular} and completes the proof of the proposition.
\end{proof}
\subsection{Unconditional well-posedness}\label{unc}

It is well known (see for instance  \cite{abfs}) that \eqref{bo} is locally well-posed in $ H^s $ for $ s>3/2 $ with a minimum time of existence which depends on
$ \|u_0\|_{H^{3/2}+} $.  As in the beginning of this section, we will use that $ u_\lambda(t,x):=\lambda^{\alpha} u(\lambda^{\alpha+1}t, \lambda x) $ is solution to \eqref{bo} with $ L_{\alpha+1}$ replaced by $ L^\lambda_{\alpha+1}$ that is the Fourier multiplier by $ i \lambda^{\alpha+1}p_{\alpha+1}(\lambda^{-1}\cdot) $. Let $ u $ be a smooth solution  to \eqref{bo} emanating from a smooth initial data $u_0$, it follows from \eqref{hyp2} that the estimate \eqref{estHsregular} also holds for $ u_\lambda $ with $ 0<\lambda\le 1$.
 Since $ \|u_\lambda(0)\|_{H^{\frac{1}{2}+}} \lesssim \lambda^{\alpha-1/2} \|u_0\|_{H^{\frac{1}{2}+}} $, a classical continuity argument ensures that
 for $ \lambda \sim (1+\|u_0\|_{H^{\frac{1}{2}+}})^{-\frac{1}{\alpha-1/2}}$, $ u_\lambda $ exists at least on $ [0,1] $ with $ \|u_\lambda \|_{L^\infty(0,1;H^{\frac{1}{2}+})}  \lesssim
   \lambda^{\alpha-1/2} \|u_0\|_{H^{\frac{1}{2}+}} $. Going back to $ u $ we obtain that it exists at least on $ [0,T] $, with
    $ T=T(\|u_0\|_{H^{\frac{1}{2}+}}):=(1+\|u_0\|_{H^{\frac{1}{2}+}})^{-\frac{2(\alpha+1)}{2\alpha-1}} $ and
   $$
   \|u\|_{L^\infty(0,T;H^s)} \lesssim \|u_0\|_{H^{s}}\mbox{ for }s>1/2\, .
   $$
 Now, let $ u_0\in H^s(\R) $ with $ s>1/2$. From the above remark, we infer that we can pass  to the limit on a sequence of solutions emanating from smooth approximations of $ u_0 $ to  obtain the existence of a solution $ u\in L^\infty_T H^s $, with initial data $ u_0 $, of \eqref{bo}. Moreover, \eqref{estdiffHsregular}-\eqref{estdiffXregular} ensure that this solution is the only one in this class.
Now the continuity of $ u $ with values in $H^{s}(\R)$ as well as the continuity of the flow-map in $H^{s}(\R)$ will follow
from the Bona-Smith argument (see \cite{BS}).
 For any $ \varphi\in H^s(\R) $, any dyadic integer $ N \ge 1$ and any $ r\ge 0 $,  straightforward calculations in Fourier space lead to
   \begin{equation}\label{init}
   \|P_{\le N} \varphi\|_{H^{s+r}_x}\lesssim N^r\|\varphi\|_{H^s_x}\quad\mbox{and} \quad
   \| \varphi-P_{\le N} \varphi\|_{H_x^{s-r}} \lesssim o(N^{-r}) \| \varphi\|_{H^s_x}\; .
   \end{equation}
   Let $ u_0 \in H^s $ with $ s>1/2$. We denote by $ u^N $ the solution of \eqref{bo} emanating from $ P_{\le N} u_0$ and
for $1\le N_1 \le N_2  $, we set
$$
w:=u^{N_1}-u^{N_2} \, .
$$
It follows from the estimates of the previous subsection applied to
$w$ that for $ T=T(\|u_0\|_{H^{\frac{1}{2}+}}) $ and any  $ -\frac{1}{2} +<r\le s $  it holds
\begin{equation}\label{to1}
\|w\|_{M_T^{r}}\lesssim \|w(0)\|_{H^{r}} \lesssim N_1^{r-s}\varepsilon(N_1)
\end{equation}
with $ \varepsilon(y) \to 0 $ as $ y\to +\infty $ . Moreover, for any $ r\ge 0 $ we have
\begin{equation}\label{to2}
\|u^{N_i}\|_{M_T^{s+r}}
\lesssim
\|u^{N_i}_0\|_{H^{s+r}}
\lesssim N_i^{r} \|u_0\|_{H^s}
\end{equation}
Next, we observe that $w$ solves the equation
\begin{equation}\label{W}
w_t+L_{\alpha+1} w  =\frac{1}{2} \partial_x (w^2) +\partial_x (u^{N_1} w)
\end{equation}
\begin{proposition}
Let $ 0<T<2$ and $ w\in L^\infty_T H^s $ with $ s>1/2 $ be a solution to \eqref{W}. Then it holds
\begin{eqnarray}
\|w\|_{L^\infty_TH^s}^2  &\lesssim & \|w_0\|_{H^s} ^2 +\|w\|_{M^s_T}^3  \nonumber \\
& & +  \|u^{N_1} \|_{M_T^s} \|w\|_{M_T^s}^2
  + \|u^{N_1} \|_{M_T^{s+1}} \|w\|_{M_T^s} \|w\|_{M_T^{s-1}}
\;.
\label{estW}
\end{eqnarray}
\end{proposition}
\begin{proof}
Actually it is a consequence  of  estimates derived in the proof of Propositions \ref{prou} and \ref{prodif}.
 We separate the contributions of $ \partial_x(w^2) $ and $ \partial_x (u^{N_1} w) $. Let $ t\in ]0,T[ $. First   \eqref{est1} leads to
 $$
 \sum_{N}N^{2s}  \Bigl| \int_0^t \int_{\R} P_N \partial_x(w^2) P_N w \Bigr| \lesssim \| w\|_{M_T^s}^3  \; .
 $$
 Second, applying \eqref{est2} at the level $s $ with $z$ replaced by $ u^{N_1} $  we obtain
 $$
  \sum_{N} N^{2s} \Bigl| \int_0^t \int_{\R} P_N \partial_x(u^{N_1} w) P_N w \Bigr| \lesssim \|u^{N_1} \|_{M_T^s} \|w\|_{M_T^s}^2
  + \|u^{N_1} \|_{M_T^{s+1}} \|w\|_{M_T^s} \|w\|_{M_T^{-\frac{1}{2}+}}
 $$
which   leads to \eqref{estW} since $ s>1/2$.
\end{proof}
 \eqref{to1}-\eqref{to2} together with \eqref{estW} lead to
\begin{eqnarray}
\|w\|_{L^\infty_TH^s}^2  &\lesssim & \|w_0\|_{H^s} ^2 +\varepsilon(N_1)
+ N_1 \,  N^{-1}_1 \varepsilon(N_1)  \label{zq}\\
 & \lesssim & \varepsilon(N_1) \nonumber    \; .
\end{eqnarray}
This shows that $ \{u^N\} $ is a Cauchy sequence in $C([0,T];H^s) $ and thus $ \{u^N\} $ converges in $C([0,T];H^s) $ to a solution of \eqref{bo} emanating from $ u_0$. Therefore, the uniqueness result ensures that $ u\in C([0,T];H^s)$.
\subsection{Continuity of the flow map} \label{cont}
Let $ s>1/2 $ and  $ \{u_{0,n}\}\subset
H^{s}(\R) $ be such that
$
u_{0,n}\rightarrow u_0
$
in $ H^{s}(\R) $.
We want to prove that the emanating solution $ u_n $ tends to $ u
$ in $ C([0,T];H^{s}) $. By the triangle inequality,
$$
\|u-u_n\|_{L^\infty_T H^{s}} \le
\|u-u^N\|_{L^\infty_T H^{s}}+
\|u^N-u^N_n\|_{L^\infty_T H^{s}}+
\|u_n^N-u_n\|_{L^\infty_T H^{s}} \quad  .
$$
Using the above estimates on the solution to \eqref{W} we first infer that
\begin{equation}\label{kak1}
\|u-u^N\|_{L^\infty_T H^{s}}+\|u_n -u_n^N \|_{L^\infty_T H^s} =\varepsilon(N),
\end{equation}
where $\varepsilon(y)\rightarrow 0$ as $y \rightarrow \infty $.
Next,  we get the bound
\begin{eqnarray}
\nonumber
\|u^N-u_n^N\|_{L^\infty_T H^{s}} & \lesssim &
\|u^N(0)-u_n^N(0)\|_{H^{s}}+o(1)
\\
\label{kak3}
& = & \|P_{\le N} (u(0)-u^{n}(0))\|_{H^{s}}+ \varepsilon(N)
\\
\nonumber
& \lesssim & \|u_0-u_{0,n}\|_{H^{s}}+ \varepsilon(N)
\end{eqnarray}
where again $\varepsilon(y)\rightarrow 0$ as $y \rightarrow \infty $.
Collecting (\ref{kak1}) and (\ref{kak3}) ends
the proof of the continuity of the flow map.
Thus the proof of Theorem~\ref{theo1} is now completed in the case $ s>1/2$.

\section{Estimates in the non regular case}\label{section4}
In this section, we provide the needed estimates at level $s\ge 1-\alpha/2$ for $1< \alpha\le 2$.
We introduce the space
\begin{equation}\label{defF}
F^{s,b} = F^{s,\alpha,b} = X^{s-\frac{\alpha+1}2,b+1/2} + X^{s-\frac{1+\alpha}{8},b+\frac{1}{8}} \;,
\end{equation}
 endowed with the usual norm  and we define
 $$
Y^s= Y^{s,\alpha} = L^\infty_t H^s\cap F^{s,\alpha,1/2}=  L^\infty_t H^s \cap (X^{s-\frac{\alpha+1}2,1} + X^{s-\frac{1+\alpha}{8},\frac{5}{8}}) \quad .
 $$
 For $ u_0\in H^s(\R) $ we will construct a solution to \eqref{bo} that belongs to $ Y^s_T $ for some $ T=T(\|u_0\|_{H^{1-\frac{\alpha}{2}}})>0 $.
  As in the regular case, by a dilation argument, we may assume that $ L_{\alpha+1} $ satisfies \eqref{hyp2} for $ 0<\lambda \le 1 $.
\begin{remark} Actually except in the case $ (s,\alpha)=(0,2) $ we could  simply take $Y^{s,\alpha}:=
  L^\infty_t H^s\cap X^{s-\frac{\alpha+1}2,1}$.
But to include the case $ (s,\alpha)=(0,2) $  in the general case we prefer to  introduce the sum space  $ F^{s,\alpha,1/2} $ (see \eqref{defF}) in all the cases.
\end{remark}
\begin{lemma}\label{YY}
Let $ 0<T<2$, $ 1<\alpha\le  2 $, $ s\ge 1-\alpha/2 $   and $ u\in L^\infty_T H^s $  be a solution to \eqref{bo}. Then
 $ u $ belongs to $ Y^{s,\alpha}_T $. Moreover,  if $ (s,\alpha)\neq (0,2) $  it holds
\begin{equation}\label{Y1}
\|u\|_{Y^{s,\alpha}_T}\lesssim \|u\|_{L^\infty_T H^s} (1+  \|u\|_{L^\infty_T H^{1-\frac{\alpha}{2}}})
\end{equation}
and if $ (s,\alpha)=(0,2) $,
\begin{equation}\label{Y2}
\|u\|_{Y^{0,2}_T}\lesssim \|u\|_{L^\infty_T L^2_x} (1+  \|u\|_{L^\infty_T L^2_x}^2) \; .
\end{equation}
\end{lemma}
\begin{proof}
As in Lemma \ref{lem1} we will work with the extension $ \tilde{u}=\rho_T u $ of $ u $ (see \eqref{defrho}). Recall that $ \supp \ut \subset [-2,2]\times \R $
 and that
 $$
 \|\ut \|_{L^\infty_t H^s} \lesssim \|u\|_{L^\infty_T H^s} \quad  \mbox{ and }\quad
  \|\ut \|_{X^{\theta,b}} \lesssim  \|u \|_{X^{\theta,b}_T}\quad ,
 $$
 for any $ (\theta,b) \in \R \times ]-\infty,1] $.
 It thus remains to control the $ F^{s,\alpha,1/2}_T $-norm of $ u $. In the case $ (s,\alpha)\neq (0,2) $ we actually  simply control the
 $X^{s-\frac{\alpha+1}2,1}_T$-norm of $ u$.
  Using the integral formulation, standard linear estimates in Bourgain's spaces  and standard product estimates in Sobolev spaces  we infer that
\begin{eqnarray*}
\|u\|_{X^{s-\frac{1+\alpha}{2},1}_T}  &\lesssim & \|u_0\|_{H^{s-\frac{1+\alpha}{2}}} + \|\partial_x(u^2)\|_{X^{s-\frac{1+\alpha}{2},0}_T}\\
& \lesssim & \|u_0\|_{H^{s-\frac{1+\alpha}{2}}} +\|u^2\|_{L^2_T H^{s+\frac{1-\alpha}{2}}} \\
 & \lesssim & \|u_0\|_{H^{s-\frac{1+\alpha}{2}}} + \|u\|_{L^\infty_T H^{1-\frac{\alpha}{2}}} \|u\|_{L^\infty_T H^{s}}\; ,
\end{eqnarray*}
since for  $ 1<\alpha\le 2 $ and $ s\ge 1-\frac{\alpha}{2}$ with $ (s,\alpha)\neq(0,2) $, it holds $ s+ 1-\frac{\alpha}{2}>0 $ and $ s+ 1-\frac{\alpha}{2}-(s+\frac{1-\alpha}{2})=1/2 $. \\
Let us now tackle  the case $ (s,\alpha)=(0,2) $. First we notice that since $ L^1(\R) \hookrightarrow H^{-\frac{3}{2}-}(\R) $, we have
\begin{equation}\label{XX}
\|u\|_{X^{-\frac{7}{4},1}_T}\lesssim \|u_0\|_{H^{-\frac{7}{4}} }+ \|u^2\|_{L^2_t H^{-\frac{3}{4}}}\lesssim \|u\|_{L^\infty_T L^2_x} (1+\|u\|_{L^\infty_T L^2_x} ) \; .
\end{equation}
To bound the $ F^{0,2,\frac{1}{2}} $-norm of $ u$, we decompose $ u^2 $ as
\begin{equation} \label{deco}
u^2= P_{\le 2} u^2+ \sum_{N>2} \Bigl( P_N(P_{\ll N} u    \, u_{\sim N}) + \sum_{N_1'\sim N_1\gtrsim N}P_N( u_{N_1} u_{N_1'})
 \Bigr) \, .
\end{equation}
The contribution of the  first term in the right hand side is easily controlled by $\|u\|_{L^\infty_T L^2_x}^2$.
  The contribution of the   $(LH)$- interactions is easily estimated  by
 \begin{align}
 \Bigl\|  \sum_{N>2}  \partial_xP_N(P_{\ll N} u   \, u_{\sim N}) \Bigr\|_{F^{0,2,-\frac{1}{2}}_T} & \lesssim
   \Bigl\|  \sum_{N>2} P_N\partial_x (P_{\ll N} u     \, u_{\sim N}) \Bigr\|_{X^{-\frac{3}{2},0}_T} \nonumber \\
& \lesssim \Bigl( \sum_{N>2} \|P_N(P_{\ll N}u    \, u_{\sim N})\|_{L^2_T L^1_x}^2\Bigr)^{1/2}  \nonumber \\
& \lesssim
\Bigl( \sum_{N\ge 1 } \| u_N\|_{L^2_T L^2_x}^2  \|P_{\ll N} u \|_{L^\infty_T L^2_x }^2 \Bigr)^{1/2} \nonumber \\
 & \lesssim   \|u\|_{L^\infty_T L^2} \|u\|_{L^\infty_T L^2_x}\label{p1} \quad .
\end{align}
To estimate the (HH)-interactions, we take advantage of the $X^{-\frac{3}{8},-\frac 38} $-part of $ F^{0,2,-\frac 12} $. For $ N>2 $ we have
\begin{multline}\label{lk1}
\sum_{N_1'\sim N_1 \gtrsim N}\|\partial_xP_N(P_{N_1} u P_{N_1'}u)\|_{F^{0,2,-\frac 12}_T} \\ \lesssim \sum_{N_1'\sim N_1\gtrsim N} N \left\|
\sum_{(L,L_1,L_2)\mbox{ \tiny{satisfying} } \eqref{resonance}} \hspace*{-10mm}\partial_xP_NQ_L(Q_{L_1}\ut_{N_1}Q_{L_2}\ut_{N_1'})\right\|_{X^{-\frac{3}{8},-\frac 38}}.
\end{multline}
For the contribution of the sum over $ L\gtrsim N N_1^2 $ in \eqref{lk1} we obtain
\begin{align}
\sum_{N_1\sim N_1'\gtrsim N} \|\partial_xP_N & Q_{\gtrsim N N_1^2}(\ut_{N_1}\ut_{N_1'})\|_{X^{-\frac{3}{8},-\frac 38}} \nonumber \\
 & \lesssim
\sum_{N_1\sim N_1'\gtrsim N} N^\frac{5}{8}N^{1/2} (N N_1^2)^{-3/8} \|\ut_{N_1} \|_{L^2_{tx}}  \|\ut_{N_1'} \|_{L^\infty_{t}L^2_x}\nonumber\\
& \lesssim \|\ut \|_{L^\infty_t L^2_x}  \sum_{N_1\gtrsim N} \Bigl( \frac{N}{N_1} \Bigr)^{3/4}  \|\ut_{N_1} \|_{L^2_{tx}}\nonumber \\
& \lesssim \gamma_N \, \|\ut \|_{L^\infty_t L^2_x}^2 \; , .\label{lk2}
\end{align}
with $ \|(\gamma_{2^j})\|_{l^2(\N)} \le 1 $.  The contribution of the region ( $ L \ll N N_1^2 $ and  $ L_1\gtrsim N N_1^2 $ )
  in \eqref{lk1} is  controlled  by
\begin{align}
\sum_{N_1\sim N_1'\gtrsim N}  \|\partial_xP_N & Q_{\ll N N_1^2} (Q_{\gtrsim N N_1^2} \ut_{N_1}\ut_{N_1'})\|_{X^{-\frac{3}{8},-\frac 38}}\nonumber \\
& \lesssim
\sum_{N_1\sim N_1'\gtrsim N} N^\frac{5}{8}N^{1/2} (N N_1^2)^{-1} N_1^{\frac{7}{4}} \|\ut_{N_1} \|_{X^{-\frac{7}{4},1}}  \|\ut_{N_1'} \|_{L^\infty_{t}L^2_x}
\nonumber\\
& \lesssim N^{-1/8}\,  \|\ut \|_{L^\infty_t L^2_x} \|\ut \|_{X^{-\frac{7}{4},1}} \; .\label{lk3}
\end{align}
Finally, the contribution of the last region  ( $ L,L_1 \ll N N_1^2 $ and  $ L_2\sim N N_1^2 )$  in \eqref{lk1} is controlled   in the same way.
Gathering \eqref{XX} and \eqref{lk1}-\eqref{lk3}, we obtain the desired result for the case $ (s,\alpha)=(0,2)$.
\end{proof}
In the sequel we will need the following straightforward estimates.
\begin{lemma}\label{lemB}
Let  $  \alpha\ge 0  $ and $ w\in F^{0,\frac{1}{2}}$.
For $ 1\le B\lesssim  N^{\alpha+1} $ it holds
\begin{equation}\label{B1}
\| Q_{\gtrsim B} w_N \|_{L^2} \lesssim B^{-1} N^{\frac{1+\alpha}{2}} \|Q_{\gtrsim B} w_N\|_{F^{0,\frac{1}{2}}}
\end{equation}
and, for $ B\gtrsim \cro{N}^{\alpha+1} $, it holds
\begin{equation}\label{B2}
\| Q_{\gtrsim B} w_N \|_{L^2} \lesssim B^{-5/8} \cro{N}^{\frac{1+\alpha}{8}} \|Q_{\gtrsim B} w_N\|_{F^{0,\frac{1}{2}}}\; .
\end{equation}
\end{lemma}
\begin{proof}
Noticing that $ F^{0,\frac 12}=F^{0,\alpha,\frac 12}= X^{-\frac{1+\alpha}{2},1} + X^{-\frac{1+\alpha}{8}, \frac{5}{8}} $, it is direct to check that
\begin{eqnarray*}
\| Q_{\gtrsim B} w_N \|_{L^2} &\lesssim &\max(B^{-1} \cro{N}^{\frac{1+\alpha}{2}}, B^{-5/8} \cro{N}^{\frac{1+\alpha}{8}}\Bigr)
\| Q_{\gtrsim B} w_N \|_{F^{0,\frac{1}{2}}}\\
& \lesssim & B^{-5/8} \cro{N}^{\frac{1+\alpha}{8}} \max\Bigl( (\cro{N}^{1+\alpha}/B)^\frac{3}{8} ,1\Bigr)  \| Q_{\gtrsim B} w_N \|_{F^{0,\frac{1}{2}}}
\end{eqnarray*}
which leads to the desired result.
\end{proof}
Now we rewrite Lemma \ref{lemtriest} in the context of the $F^{s,b}$ spaces.
\begin{lemma}\label{lemtriestY}
Assume $u_i\in Y^0$, $i=1,2,3$ are functions with spatial Fourier support in $\{|\xi|\sim N_i\}$ with $N_i>0$ dyadic satisfying $N_1\le N_2\le N_3$. \\
If $N_3>2^9$ and $N_1\gtrsim N_3^{\frac 23(1-\alpha)}\wedge 2^9$, it holds for $(p,q)\in \{(2,\infty), (\infty,2)\}$
\begin{align*}
|I_t(u_1,u_2,u_3)|\lesssim & \sum_{l\ge -4} 2^{-l} N_1^{-1/2} N_3^{\frac{1-\alpha}2} \|u_1\|_{L^p_tL^2_x} \|Q_{2^lN_1N_3^\alpha}u_2\|_{F^{0,\frac 12}} \|u_3\|_{L^q_tL^2_x}\\
&+ N_1^{-1/2} N_3^{\frac{1-\alpha}2} \|u_1\|_{L^p_tL^2_x} \|u_2\|_{L^q_tL^2_x} \|Q_{\sim N_1N_3^\alpha}u_3\|_{F^{0,\frac 12}}\\
&+ N_1^{-1/8}\cro{N_1}^{\frac{1+\alpha}8}N_3^{-5\alpha/8} \|u_1\|_{F^{0,\frac 12}} \|u_2\|_{L^2_{tx}} \|u_3\|_{L^\infty_tL^2_x}\\
&+ N_1^{-1/4}N_3^{\frac 18-\frac \alpha 2} \|u_1\|_{L^\infty_tL^2_x} \|u_2\|_{L^\infty_tL^2_x} \|u_3\|_{L^\infty_tL^2_x}.
\end{align*}
\end{lemma}
\begin{proof}
For $R=N_1^{3/4}N_3^{\frac \alpha 2-\frac 18}$ we decompose $I_t$ as in (\ref{decIt}) and obtain from (\ref{estIthigh}) that
$$
|I_t^{high}| \lesssim N_1^{-1/4} N_3^{\frac 18-\frac \alpha 2} \prod_{i=1}^3 \|u_i\|_{L^\infty_tL^2_x}.
$$
To evaluate $I_t^{low}$ we use decomposition (\ref{AA}) and notice that
$$
R = N_1^{3/4}N_3^{\frac \alpha 2-\frac 18} \le N_1 N_2^{\frac{2\alpha}3 - \frac 7{24}}\ll N_1N_3^\alpha \textrm{ and } N_1 N_3^{\alpha}\gtrsim N_3^{\frac{2+\alpha}3}\gg 1.
$$
Therefore the  contribution  $  I_t^{1,low} $ of the first term of the RHS of (\ref{AA})  to $I_t^{low} $ is easily estimated thanks to Lemmas \ref{ilow-lem} and \ref{lemB} by
$$
 |I_t^{1,low}|    \lesssim   N_1^{1/2}(N_1N_3^\alpha)^{-5/8}\cro{N_1}^{\frac{\alpha+1}8} \|u_1 \|_{F^{0,\frac 12} } \|u_2 \|_{L^2_{tx}} \|u_3 \|_{L^\infty_t L^2_x} ,
 $$
 which is acceptable. Thanks to Lemmas \ref{continuiteQ}, \ref{ilow-lem} and \ref{lemB}, the contribution $ I_t^{2,low} $ of the second term can be handle in the following way
  \begin{align}
 |I_t^{2,low}| \lesssim &  \sum_{l\ge -4} N_1^{1/2}(2^lN_1 N_3^\alpha)^{-1} N_3^{\frac{\alpha+1}2}
 \|u_1 \|_{L^p_t L^2_x } \|Q_{2^lN_1N_3^\alpha} u_2 \|_{F^{0,\frac 12}} \|u_3 \|_{L^q_tL^2_x} \nonumber \\
   \lesssim & \sum_{l\ge -4} 2^{-l} N_1^{-1/2}N_3^{\frac{1-\alpha}2} \|u_1 \|_{L^p_t L^2_x } \|Q_{2^lN_1N_3^\alpha} u_2 \|_{F^{0,\frac 12}} \|u_3 \|_{L^q_tL^2_x} \; .
\end{align}
In the same way, we get that the contribution $I_t^{3,low}$ of the third term to $I_t^{low}$ is bounded by
\begin{align}
 |I_t^{3,low}| \lesssim &   N_1^{1/2}(N_1 N_3^\alpha)^{-1} N_3^{\frac{\alpha+1}2}
 \|u_1 \|_{L^p_t L^2_x } \| u_2 \|_{L^q_tL^2_x} \|Q_{\sim N_1N_3^\alpha} u_3 \|_{F^{0,\frac 12}} \nonumber \\
   \lesssim & N_1^{-1/2}N_3^{\frac{1-\alpha}2} \|u_1 \|_{L^p_t L^2_x } \| u_2 \|_{L^q_tL^2_x} \|Q_{\sim N_1N_3^\alpha} u_3 \|_{F^{0,\frac 12}} \; .
\end{align}
Gathering all these estimates, we obtain the desired bound.
\end{proof}
\begin{proposition} \label{YYY}
Let $ 0<T<2$, $ 1<\alpha\le 2 $, $ s\ge 1-\alpha/2 $   and $ u\in L^\infty_T H^s $  be a solution to \eqref{bo}. Then  $ u $ belongs to
 $ \widetilde{L^\infty_T} H^s $ and it holds
\begin{equation}\label{estHsnonregular}
\|u\|_{\widetilde{L^\infty_T} H^s}  \lesssim \|u_0\|_{H^s}  +  \|u\|_{L^\infty_T H^{1-\frac{\alpha}{2}}} \|u\|_{Y^s_T} + \|u\|_{L^\infty_T H^s}\|u\|_{Y_T^{1-\frac{\alpha}{2}}}
 \;.
\end{equation}
\end{proposition}
\begin{proof}
Applying the operator $ P_N $ with $ N>0  $ dyadic  to equation \eqref{bo}, taking the $H^s$ scalar product with $ P_N u $ and integrating
 on $ ]0,t[ $  we obtain
\begin{equation}\label{415}
\|P_N u\|^2_{L^\infty_T H^s} \lesssim \|P_N u_0\|_{H^s}^2 + \sup_{t\in ]0,T[}  \langle N\rangle^{2s} \Bigl| \int_0^t \int_{\R}  P_N(u^2) \partial_xP_N u \Bigr| .
\end{equation}
We take an extension  $\tilde{u} $ of $ u $ supported in time in $ ]-4,4[ $  such that $ \|\tilde{u}\|_{Y^s} \lesssim \|u\|_{Y^s_T} $ . To simplify the notation we drop the tilde in the sequel.\\
We infer from (\ref{estI1}) that it suffices to estimate
$$
I = \sum_{N>0}\sum_{N_1\gtrsim  N}  N\langle N_1\rangle^{2s}\sup_{t\in ]0,T[} |I_t(u_N, u_{\sim N_1}, u_{N_1})|.
$$
The low frequencies part $N\le 2^9$ is estimated exactly as in (\ref{estI2}) by
$$
\|u\|_{L^\infty_tL^2_x}\|u\|_{L^\infty_tH^s}^2.
$$
On  the other hand, the contribution of the sum over $N>2^9$ is controlled thanks to Lemma \ref{lemtriestY} by
\begin{align}
\sum_{N>2^9}\sum_{N_1\gtrsim N}& \Big[ \left(\frac N{N_1}\right)^{\frac{\alpha-1}2} \|u_N\|_{L^2_tH^{1-\frac \alpha 2}} \|u_{N_1}\|_{L^\infty_tH^s} \|u_{N_1}\|_{F^{s,\frac 12}} \nonumber \\
&\quad + \left(\frac N{N_1}\right)^{5\alpha/8} \|u_N\|_{F^{1-\frac \alpha 2,\frac 12}} \|u_{N_1}\|_{L^2_tH^s} \|u_{N_1}\|_{L^\infty_t H^s}\nonumber \\
&\quad + N^{\frac \alpha 2-\frac 14} N_1^{\frac 18-\frac \alpha 2} \|u_N\|_{L^\infty_tH^{1-\frac \alpha 2}} \|u_{N_1}\|_{L^\infty_t H^s}^2 \Big]\ \nonumber\\
&\lesssim \|u\|_{Y^{1-\frac{\alpha}{2}}} \|u\|_{L^\infty_t H^s}^2 + \|u\|_{L^\infty_t H^{1-\frac{\alpha}{2}}}
 \|u\|_{L^\infty_t H^s} \|u\|_{Y^s}, \label{416}
\end{align}
where we use discrete Young's inequality in $N_1$ and then Cauchy-Schwarz in $N$ to bound the first two terms.\\
Gathering the above estimates we eventually obtain
\begin{equation}\label{N}
I\lesssim \|u\|_{Y^{1-\frac{\alpha}{2}}_T} \|u\|_{L^\infty_T H^s}^2 + \|u\|_{L^\infty_T H^{1-\frac{\alpha}{2}}}
 \|u\|_{L^\infty_T H^s} \|u\|_{Y^s_T} \, ,
\end{equation}
which  completes the proof of the proposition.

\end{proof}
\subsection{Estimates on the difference of two solutions}
First we introduce the function spaces where we will estimate the difference of two solutions of \eqref{bo}. Contrary to the regular case, we will have to work in a function space that put a weight on the very low frequencies.
For $ \theta\in \R $ we denote by $ \overline{H}^\theta $ the completion of $  {\mathcal S}(\R) $  for the norm
$$
\|\varphi\|_{\overline{H}^\theta} := \| \langle |\xi|^{-1/2}\rangle \langle \xi \rangle^\theta \hat{\varphi} \|_{L^2}
$$
Then we define the space  $ \widetilde{L^{\infty}_t} \overline{H}^{\theta} $ by
\begin{equation}
\|w\|_{\widetilde{L^{\infty}_t}\overline{H}^{\theta}} :=\Bigr( \sum_{N-dyadic>0}  \|w_N \|_{L^\infty_t  \overline{H}^\theta}^2 \Bigr)^{1/2} \; .
\end{equation}
We then define the function spaces $ \tilde{Y}^\theta $ and $ Z^\theta $, $ \theta\in\R $, by respectively
$$
 \tilde{Y}^\theta= \widetilde{L^\infty_t} H^\theta\cap F^{\theta,1/2} \;  \mbox{ and }
Z^\theta = \widetilde{L^\infty_t}\overline{H}^\theta\cap F^{\theta,1/2} \; ,
$$
with $ F^{\theta,b} $ defined in \eqref{defF}.

If $u,v\in L^\infty_T H^s $ are two solutions of (\ref{bo}) with $s\ge 1-\alpha/2$,  then according to  Lemma \ref{YY} and Proposition \ref{YYY} we know that
 $ u $ and $ v$ belong to $ Y^s_T \cap  \widetilde{L^\infty_T} H^s $. Moreover, using again the extension operator $ \rho_T $ it is easy to check that
 \begin{equation}\label{tilde}
 Y^s_T \cap  \widetilde{L^\infty_T} H^s \hookrightarrow \tilde{Y}^s_T
 \end{equation}
 with an embedding constant that does not depend on $ 0<T\le 2 $. Hence, $ u $ and $ v $ belong to $\tilde{Y}^s_T$. Assuming that $ u_0-v_0 \in 
  \overline{H}^s $, we claim  that the difference $ u-v $ belongs to $ Z^s_T $. Indeed, according to the above definitions of
   $ \tilde{Y}^s $ and $ Z^s $ , it suffices to check that $ P_{1} (u-v) $ belongs to $ \widetilde{L^\infty_T} \overline{H}^s$. But it is straightforward, since by the Duhamel formula for any dyadic integer $ 0<N<1 $ it holds 
   $$
   \|P_N(u-v)\|_{L^\infty_T \overline{H}^s} \lesssim \|u_0-v_0 \|_{  \overline{H}^s}+ N (\|u\|^2_{L^\infty_T L^2_x} + \|v\|^2_{L^\infty_T L^2_x}) \; .
   $$
 We are thus allowed to estimate the difference $w=u-v$ in the space $Z^{s-\frac{3}{2}+\frac{\alpha}{2}}_T$. 
\begin{proposition}\label{prop-wF}
Let $ 0<T<1$, $1<\alpha\le 2$, $s\ge 1-\alpha/2$ and $u,v \in L^\infty_T H^s $  be two solutions to (\ref{bo}) on $]0,T[$. Then we have
\begin{equation}\label{wF1}
\|u-v\|_{Z^{s-\frac{3}{2}+\frac{\alpha}{2}}_T} \lesssim \|u-v\|_{L^\infty_T \overline{H}^{s-\frac{3}{2}+\frac{\alpha}{2}}} + \|u+v\|_{\tilde{Y}^s_T} \|u-v\|_{Z^{-1/2}_T}
+ \|u+v\|_{\tilde{Y}^{1-\frac{\alpha}{2}}}\|u-v\|_{Z^{s-\frac{3}{2}+\frac{\alpha}{2}}_T}\, .
\end{equation}
\end{proposition}

\begin{proof}
Recall that $ w=u-v $ satisfies  \eqref{eq-diff} with $z=u+v$.
We extend $ w $ from $ (0,T) $ to $ \R$ by using the extension operator $ \rho_T $ defined in \eqref{defrho}. On account of the uniform bounds on $
\rho_T $ (see the paragraph just after \eqref{defrho}), it  remains to estimate the $ F^{s-\frac{3}{2}+\frac{\alpha}{2},\alpha,\frac{1}{2}}_T $-norm  of $ w $.
From classical linear estimates in the framework of Bourgain's spaces, the Duhamel formulation associated with (\ref{eq-diff})  leads  to
\begin{equation}\label{est-wF}
\|w\|_{F^{s-\frac{3}{2}+\frac{\alpha}{2}, 1/2}_T} \lesssim \|w_0\|_{H^{s-\frac{3}{2}+\frac{\alpha}{2}}} + \|\partial_x(zw)\|_{F^{s-\frac{3}{2}+\frac{\alpha}{2},-1/2}_T}.
\end{equation}
Let $\zt$ and $\wt$ be time extensions of $z$ and $w$ satisfying $\|\zt\|_{\tilde{Y}^s}\lesssim \|z\|_{\tilde{Y}^s_T}$ and $\|\wt\|_{Z^{s-\frac{3}{2}+\frac{\alpha}{2}}}\lesssim \|w\|_{Z^{s-\frac{3}{2}+\frac{\alpha}{2}}_T}$. To simplify the notation we drop the tilde in the sequel. From (\ref{est-wF}) we see that it suffices to estimate
$$
\|\partial_x(zw)\|_{F^{s-\frac{3}{2}+\frac{\alpha}{2},\alpha,-1/2}} \lesssim \left(\sum_N \|P_N\partial_x (zw)\|_{F^{s-\frac{3}{2}+\frac{\alpha}{2},-1/2}}^2\right)^{1/2}.
$$
We first estimate the (low-high) contribution  $P_N(P_{\lesssim N}z P_Nw)$:
\begin{align*}
\|\partial_xP_N(P_{\lesssim N}z P_Nw)\|_{F^{s-\frac{3}{2}+\frac{\alpha}{2},\alpha,-1/2}} &\lesssim  \sum_{N_1\lesssim N}N\|P_N(P_{N_1}z P_Nw)\|_{X^{s-2,0}}\\
&\lesssim \sum_{N_1\lesssim N}N_1^{1/2}N \cro{N}^{s-2} \|P_{N_1}z\|_{L^\infty_t L^2_x} \|P_Nw\|_{L^2_t L^2_x}\\
&\lesssim  \|P_Nw\|_{L^2_t H^{s-\frac{3}{2}+\frac{\alpha}{2}}}
\sum_{N_1\lesssim N}\Bigl( \frac{N_1}{\cro{N}}\Bigr)^{\frac{\alpha-1}{2}} \|P_{N_1} z\|_{L^\infty_tH^{1-\frac{\alpha}{2}}}\\
&\lesssim \|z\|_{L^\infty_tH^{1-\frac{\alpha}{2}}} \|P_Nw\|_{L^\infty_t {H}^{s-\frac{3}{2}+\frac{\alpha}{2}}}.
\end{align*}
Similarly, the (high-low) interactions are estimated as follows:
\begin{align*}
\|\partial_xP_N(P_Nz P_{\lesssim N}w)\|_{F^{s-\frac{3}{2}+\frac{\alpha}{2},-1/2}} &\lesssim N\|P_N(P_Nz P_{\lesssim N}w)\|_{X^{s-2,0}}\\
&\lesssim   \|P_Nz\|_{L^2_tH^s}  \sum_{N_1\lesssim N} \Bigl( \frac{N_1}{\cro{N}}\Bigr)^{1/2}  \|P_{N_1}w\|_{L^\infty_t H^{-1/2}}\\
&\lesssim \|P_Nz\|_{L^2_tH^s} \|w\|_{L^\infty_t {H}^{-\frac{1}{2}}}.
\end{align*}
Now we deal with the high-high interactions term
\begin{multline*}
\|\partial_xP_N(P_{\gg N}z P_{\gg N}w)\|_{F^{s-\frac{3}{2}+\frac{\alpha}{2},-\frac 12}} \\ \lesssim \sum_{N_1\gg N} N \left\|
\sum_{(L,L_1,L_2)\mbox{ \tiny{satisfying} } \eqref{resonance}} \hspace*{-10mm}\partial_xP_NQ_L(Q_{L_1}z_{N_1}Q_{L_2}w_{N_1})\right\|_{F^{s-\frac{3}{2}+\frac{\alpha}{2},-\frac 12}},
\end{multline*}
We may assume that $N_1\gg 1$ since otherwise, it holds $N\ll N_1\lesssim 1$ and we have
$$
\|P_{\lesssim 1}\partial_x(P_{\lesssim 1}z P_{\lesssim 1}w)\|_{F^{s-\frac{3}{2}+\frac{\alpha}{2},-1/2}} \lesssim
 \|P_{\lesssim 1}z\|_{L^\infty_t L^2} \|P_{\lesssim 1}w\|_{L^\infty_t H^{-\frac{1}{2}}}.
$$
For  $ N_1 \gg 1 $, we will  take advantage of the fact that $ X^{s-\frac{13}{8}+\frac{3\alpha}{8},-\frac{3}{8}}\hookrightarrow F^{s-\frac{3}{2}+\frac{\alpha}{2},-1/2}$. The  contribution of the sum over   $L\gtrsim NN_1^\alpha$ can be thus controlled by
\begin{align*}
&\sum_{N_1\gg N}\|\partial_xP_NQ_{\gtrsim NN_1^\alpha}(z_{N_1} w_{N_1})\|_{F^{s-\frac{3}{2}+\frac{\alpha}{2},-1/2}} \\
&\quad \lesssim \sum_{N_1\gg N} N \|P_NQ_{\gtrsim NN_1^\alpha}(z_{N_1} w_{N_1})\|_{X^{s-\frac{13}{8}+\frac{3\alpha}{8},-3/8}} \\
&\quad \lesssim \sum_{N_1\gg N}\sum_{L\gtrsim NN_1^\alpha} N \cro{N}^{s-\frac{13}{8}+\frac{3\alpha}{8}} L^{-3/8} \|P_NQ_L(z_{N_1}w_{N_1})\|_{L^2}\\
&\quad \lesssim \sum_{N_1\gg N} N^{3/2} \cro{N}^{s-\frac{13}{8}+\frac{3\alpha}{8}} (NN_1^\alpha)^{-3/8} N_1^{\frac{1}{2}-s}\|z_{N_1}\|_{L^2_tH^s} \|w_{N_1}\|_{L^\infty_t H^{-1/2}}\\
&\quad \lesssim \sum_{N_1\gg N} (N/N_1)^{1/2-\alpha/8} \Bigl( \frac{\cro{N}}{\cro{N_1}}\Bigr)^{s-1+\frac{\alpha}{2}}   \|z_{N_1}\|_{L^2_tH^s} \|w_{N_1}\|_{L^\infty_tH^{-1/2}} \\
& \quad \lesssim  \delta_N \| z\|_{L^2_t H^s}  \|w\|_{L^\infty_t H^{-1/2}},
\end{align*}
where $ \|(\delta_{2^j})_{j} \|_{l^2(\Z)} \lesssim 1 $.
The contribution of  the region ($ L \ll NN_1^\alpha $  and $L_1\gtrsim NN_1^\alpha$) is estimated thanks to \eqref{B1} by
\begin{align*}
&\sum_{N_1\gg N}\|\partial_xP_NQ_{\ll NN_1^\alpha}(Q_{\gtrsim NN_1^\alpha}z_{N_1} w_{N_1})\|_{X^{s-\frac{13}{8}+\frac{3\alpha}{8},-3/8}} \\
&\quad \lesssim \sum_{N_1\gg N} N \cro{N}^{s-\frac{13}{8}+\frac{3\alpha}{8}} \|P_N(Q_{\gtrsim NN_1^\alpha}z_{N_1}w_{N_1})\|_{L^2}\\
&\quad \lesssim \sum_{N_1\gg N}  N^{3/2} \cro{N}^{s-\frac{13}{8}+\frac{3\alpha}{8}}(N N_1^\alpha)^{-1}N_1^{1-s+\frac{\alpha}{2}}
\|Q_{\gtrsim NN_1^\alpha}z_{N_1}\|_{F^{s,\frac{1}{2}}} \|w_{N_1}\|_{L^\infty_tH^{-1/2}} \\
&\quad \lesssim \sum_{N_1\gg N}\Bigl( \frac{N}{\cro{N}}\Bigr)^{1/2} \cro{N}^{-\frac{1+\alpha}{8}}
  \Bigl( \frac{\cro{N}}{\cro{N_1}}\Bigr)^{s-1+\frac{\alpha}{2}}
\|Q_{\gtrsim NN_1^\alpha}z_{N_1}\|_{F^{s,\frac{1}{2}}} \|w_{N_1}\|_{L^\infty_tH^{-1/2}} \\
 & \quad \lesssim  \delta_N  \| z\|_{Y^s} \|w\|_{\widetilde{L^\infty_t}H^{-1/2}}
\end{align*}
where $ \|(\delta_{2^j})_{j} \|_{l^2(\Z)} \lesssim 1 $.
Finally   the contribution of the last region  can be bounded thanks to \eqref{B1} by
\begin{align*}
&\sum_{N_1\gg N}\|\partial_xP_NQ_{\ll NN_1^\alpha}(Q_{\ll N N_1^{\alpha}} z{N_1}\, Q_{\sim NN_1^\alpha}w_{N_1})\|_{X^{s-\frac{13}{8}+\frac{3\alpha}{8},-\frac{3}{8}}} \\
&\quad \lesssim \sum_{N_1\gg N} N \cro{N}^{s-\frac{13}{8}+\frac{3\alpha}{8}} \|P_NQ_{\ll NN_1^\alpha}(Q_{\ll N N_1^{\alpha}} z_{N_1}\, Q_{\sim NN_1^\alpha}w_{N_1})\|_{L^2}\\
&\quad \lesssim \sum_{N_1\gg N}  N^{3/2} \cro{N}^{s-\frac{13}{8}+\frac{3\alpha}{8}} N_1^{-s} (N N_1^\alpha)^{-1} N_1^{1+\alpha/2}
\|Q_{\ll NN_1^\alpha}z_{N_1}\|_{L^\infty_t H^s} \|Q_{\sim NN_1^\alpha} w_{N_1}\|_{F^{-1/2,1/2}} \\
&\quad \lesssim \sum_{N_1\gg N}\Bigl( \frac{N}{\cro{N}}\Bigr)^{1/2} \cro{N}^{-\frac{1+\alpha}{8}}
  \Bigl( \frac{\cro{N}}{\cro{N_1}}\Bigr)^{s-1+\frac{\alpha}{2}}
\|z_{N_1}\|_{L^\infty_t H^s} \|w_{N_1}\|_{F^{-1/2,1/2}} \\
& \quad \lesssim  \delta_N \|z\|_{\widetilde{L^\infty_t} H^s} \|w\|_{Z^{-1/2}}
\end{align*}
which is acceptable. This concludes the proof of Proposition \ref{prop-wF}.
\end{proof}

\begin{proposition}\label{PP1}
Let $ 1\le \alpha \le 2 $, $ 0<T<2$ and $ u,v \in L^\infty_T H^s $ with $ s\ge 1-\alpha/2 $ be two solutions to \eqref{bo}. Then it holds\footnote{ We include the case $ \alpha=1 $ here  since it does not lead to additional difficulties and  will be useful  in the appendix to prove LWP for $ (\alpha,s)=(1,1/2)$.}
\begin{equation}\label{estdiffHscritical}
\|u-v\|_{\widetilde{L^{\infty}_T}\overline{H}^{s-\frac{3}{2}+\frac{\alpha}{2}} }  \lesssim \|u_0-v_0\|_{\overline{H}^{s-\frac{3}{2}+\frac{\alpha}{2}} }  +
\|u+v\|_{Y^s_T}
\|u-v\|_{Z^{s-\frac{3}{2}+\frac{\alpha}{2}}_T} \;.
\end{equation}
\end{proposition}
\begin{proof}
 Recall that the difference $w=u-v$ satisfies \eqref{eq-diff} with $ z=u+v$. \\
 Applying the operator $ P_N $ with $ N>0  $ dyadic  to equation \eqref{eq-diff}, taking the $H^s$ scalar product with $ P_N w $ and integrating
 on $ ]0,t[ $  we obtain
$$
\|w_N\|_{L^\infty_T \overline{H}^{s-\frac{3}{2}+\frac{\alpha}{2}}}^2\lesssim \|P_N w_0\|_{\overline{H}^{s-\frac{3}{2}+\frac{\alpha}{2}}}
+ \cro{N^{-1}} \cro{N}^{2(s-\frac 32+\frac \alpha 2)} \sup_{t\in [0,T]} \Bigl| \int_0^t\int_{\R}  P_N ( zw) \partial_x w_N \Bigr|.
$$
Therefore we have to estimate
$$
J:= \sum_{N>0} \cro{N^{-1}} \cro{N}^{2(s-\frac {3-\alpha}2)} \sup_{t\in [0,T]} \Bigl| \int_0^t\int_{\R}  P_N ( zw) \partial_x w_N \Bigr|.
$$
We  take extensions  $\tilde{z} $ and $ \tilde{w} $  of $ z $ and $ w$ supported in time in $]-4,4[ $   such that $ \|\tilde{z}\|_{Y^s} \lesssim \|u\|_{Y^s_T} $
 and  $ \|\tilde{w}\|_{Z^{s}} \lesssim \|w\|_{Z^{s}_T} $. To simplify the notation we drop the tilde in the sequel.\\
Proceeding as in (\ref{J1}) we get
\begin{align}
J &\lesssim \sum_{N>0}\sum_{N_1\gtrsim N} N\cro{N_1^{-1}} \langle N_1\rangle^{2(s-\frac {3-\alpha}2)} \sup_{t\in ]0,T[} |I_t(z_N, w_{\sim N_1}, w_{N_1})| \nonumber \\
&\quad + \sum_{N>0}\sum_{N_1\gtrsim N} N_1 \cro{N_1^{-1}} \langle N_1\rangle^{2(s-\frac {3-\alpha}2)} \sup_{t\in ]0,T[} |I_t(z_{\sim N_1}, w_N, w_{N_1})| \nonumber \\
&\quad + \sum_{N>0}\sum_{N_1\gtrsim N} N \cro{N^{-1}} \langle N\rangle^{2(s-\frac {3-\alpha}2)} \sup_{t\in ]0,T[} |I_t(z_{N_1}, w_{N_1}, w_N)| \nonumber \\
&:= J_1+J_2+J_3 \label{J2}.
\end{align}
{\it Estimates for $J_1$.} The contribution of the sum over $N\le 2^9$ in $J_1$ is estimated thanks to (\ref{estitlow}) by
\begin{align*}
&\sum_{N\le 2^9}\sum_{N_1\gtrsim N} N^{3/2} \|z_N\|_{L^\infty_tL^2_x} \|w_{N_1}\|_{L^\infty_t \overline{H}^{s-\frac{3-\alpha}2}}^2\\
&\quad \lesssim \|z\|_{L^\infty_tL^2_x} \|w\|_{\widetilde{L^\infty_T} \overline{H}^{s-\frac{3-\alpha}2}}^2.
\end{align*}
The contribution $N>2^9$ in $J_1$ can be controlled with Lemma \ref{lemtriestY} by
\begin{align*}
\sum_{N>2^9} \sum_{N_1\gtrsim N}& \Big[ \sum_{l\ge -4}2^{-l} \left(\frac N{N_1}\right)^{\frac{\alpha-1}2} \|z_N\|_{L^2_tH^{1-\frac \alpha 2}} \|Q_{2^l NN_1^\alpha}w_{N_1}\|_{F^{s-\frac {3-\alpha}2, \frac 12}} \|w_{N_1}\|_{L^\infty_t H^{s-\frac{3-\alpha}2}}\\
&\quad + \left(\frac N{N_1}\right)^{5\alpha/8} \|z_N\|_{F^{1-\frac \alpha 2,\frac 12}} \|w_{N_1}\|_{L^2_tH^{s-\frac{3-\alpha}2}} \|w_{N_1}\|_{L^\infty_tH^{s-\frac{3-\alpha}2}}\\
&\quad + N^{\frac \alpha 2-\frac 14} N_1^{\frac 18-\frac \alpha 2} \|z_N\|_{L^\infty_t H^{1-\frac \alpha 2}} \|w_{N_1}\|_{L^\infty_tH^{s-\frac{3-\alpha}2}}^2\Big]\\
&\lesssim \|z\|_{Y^{1-\frac \alpha 2}} \|w\|_{\widetilde{L^\infty_t}\overline{H}^{s-\frac{3-\alpha}2}} \|w\|_{Z^{s-\frac{3-\alpha}2}}
\end{align*}
where we used for the first term Cauchy-Schwarz in $(N,N_1)$ and then sum in $l$. Note that for $ \alpha>1 $ we could replace the $ \widetilde{L^\infty_t} H^{s-\frac{3}{2}+\frac{\alpha}{2}}$-norm by a standard $L^\infty_t H^{s-\frac{3}{2}+\frac{\alpha}{2}} $ by invoking the discrete Young inequality. \\
{\it Estimates for $J_2$.} We separate different contributions.  First, the contribution of the sum over $ N_1 \le 2^9 $ is directly estimated by
 $ \|z\|_{L^\infty_T L^2} \|w\|_{L^\infty_T H^{-\frac{1}{2}}}^2$.
The contribution of the sum over $ N\le N_1^{\frac{2}{3}(1-\alpha)} $ and $N_1>2^9$ is then easily estimated by
\begin{align}
  \sum_{N_1 >  2^9} &\sum_{N\le N_1^{\frac{2}{3}(1-\alpha)}  }  N N_1^{\frac{\alpha-1}2} \| z_{N_1}\|_{ L^2_T H^s}
 \| w_N\|_{L^\infty_T \overline{H}^{-1/2}}  \| w_{N_1}\|_{L^2_T H^{s-\frac{3}{2}+\frac{\alpha}{2}}}
 \nonumber \\
 & \lesssim \sum_{N_1 >  2^9}  N_1^{\frac{1-\alpha}6}
 \| z_{N_1}\|_{ L^2_T H^s}
 \| w\|_{L^\infty_T \overline{H}^{-1/2}}  \| w_{N_1}\|_{L^2_T H^{s-\frac{3}{2}+\frac{\alpha}{2}}}
 \nonumber \\
& \lesssim   \|z\|_{L^\infty_T H^s } \| w\|_{L^\infty_T \overline{H}^{-1/2}} \| w\|_{L^\infty_T H^{s-\frac{3}{2}+\frac{\alpha}{2}}} \; .
\end{align}
Finally the contribution of the sum over $N_1>2^9$ and $N\gg N_1^{\frac 23(1-\alpha)}$ is bounded thanks to Lemma \ref{lemtriestY} by
\begin{align*}
\sum_{N_1>2^9}\sum_{N\gg N_1^{\frac 23(1-\alpha)}} &\Big[ \sum_{l\ge -4} \|w_N\|_{L^\infty_t\overline{H}^{-1/2}}  \|Q_{2^lNN_1^\alpha} w_{N_1}\|_{F^{s-\frac{3-\alpha}2, \frac 12}}  \|z_{N_1}\|_{L^2_tH^s}\\
&\quad + \|w_N\|_{L^\infty_t\overline{H}^{-1/2}}  \|w_{N_1}\|_{L^2_tH^{s-\frac{3-\alpha}2}}  \|Q_{\sim NN_1^\alpha} z_{N_1}\|_{F^{s,\frac 12}}\\
&\quad + N^{-1/8}\cro{N}^{\frac{5+\alpha}8} N_1^{-\frac \alpha 8-\frac 12} \|w_N\|_{F^{-\frac 12,\frac 12}} \|w_{N_1}\|_{L^\infty_tH^{s-\frac{3-\alpha}2}}  \|z_{N_1}\|_{L^2_tH^s}\\
&\quad + N^{1/4} N_1^{-3/8} \|w_N\|_{L^\infty_t\overline{H}^{-1/2}} \|w_{N_1}\|_{L^\infty_tH^{s-\frac{3-\alpha}2}}  \|z_{N_1}\|_{L^\infty_tH^s}\\
&\lesssim \|z\|_{Y^s} (\|w\|_{\widetilde{L^\infty_t}\overline{H}^{-1/2}} \|w\|_{Z^{s-\frac{3-\alpha}2}} + \|w\|_{Z^{-1/2}} \|w\|_{\widetilde{L^\infty_t}\overline{H}^{-1/2}})
\end{align*}
where again we used Cauchy-Schwarz in $(N,N_1)$ and then sum over $l$.\\
{\it Estimates for $J_3$.} We first notice that for $N\lesssim N_1$ and $N_1>2^9$, since $1+2(s-\frac{3-\alpha}2)\ge 0$, it holds
$$
N\cro{N^{-1}} \cro{N}^{2(s-\frac{3-\alpha}2)} \lesssim N_1\cro{N_1^{-1}} \cro{N_1}^{2(s-\frac{3-\alpha}2)}.
$$
Therefore the contribution of this region to $J_3$ is controlled by $J_2$. Finally the contribution of $N\lesssim N_1\le 2^9$ is easily bounded by $\|z\|_{L^\infty_t L^2_x} \|w\|_{L^\infty_t \overline{H}^{-1/2}}^2$.\\

Gathering all the estimates, we eventually obtain
\begin{equation}\label{NN}
J\lesssim      \|z\|_{Y^s} \|w\|_{L^\infty_T \overline{H}^{-1/2} }  \|w \|_{Z^{s-\frac{3}{2}+\frac{\alpha}{2}}_T}+   \|z\|_{Y^{1-\frac{\alpha}{2}}_T}
  \|w\|_{\widetilde{L^\infty_T} H^{s-\frac{3}{2}+\frac{\alpha}{2}}}  \|w \|_{Z^{s-\frac{3}{2}+\frac{\alpha}{2}}_T}
\end{equation}

which completes the proof of  \eqref{estdiffHscritical}.
\end{proof}
\subsection{Unconditional well-posedness}\label{section42}
We argue as in Section \ref{unc}. We notice that $ 1-\frac{\alpha}{2} \ge 0 >s_c = \frac{1}{2}-\alpha $ that is the critical Sobolev exponent associated with  \eqref{bo}
 for dilation symmetry. Estimates \eqref{Y1}, \eqref{Y2}, \eqref{estHsnonregular}, dilations and continuity arguments ensure that  the  time of existence of a smooth solution is bounded from below by
 $ T=T(\|u_0\|_{H^{1-\frac{\alpha}{2}}})\sim  (1+ \|u_0\|_{H^{1-\frac{\alpha}{2}}})^{-\frac{2(\alpha+1)}{2\alpha-1}}$. Passing to the limit on a sequence of smooth solutions we construct a solution $ u\in \tilde{Y}^s_T $ to \eqref{bo} emanating from $ u_0\in H^s(\R) $.  On the other hand, Lemma \ref{YY}, Proposition \ref{YYY} and \eqref{tilde} ensure that any $ L^\infty_T H^s $-solution to \eqref{bo} on $ ]0,T[ $ belongs to $ \widetilde{Y}^s_T$. Therefore,  according to \eqref{wF1} and  \eqref{estdiffHscritical}, $u $  is the only solution emanating from $ u_0$ that belongs to $ L^\infty_{loc} H^s $.
Now the continuity of $ u $ with values in $H^{s}(\R)$ as well as the continuity of the flow-map in $H^{s}(\R^2)$ will follow
from the Bona-Smith argument.
   Let $ u_0 \in H^s $ with $ s\ge 1-\frac{\alpha}{2}$. We denote by $ u^N $ the solution of \eqref{bo} emanating from $ P_{\le N} u_0$ and we set
for $1\le N_1 \le N_2  $, we set
$$
w:=u^{N_1}-u^{N_2} \, .
$$
Let us notice that $ P_{\le 1 } w_0 = P_{\le 1 } (u^{N_1}-u^{N_2}) =0 $ and thus $ \|w_0\|_{\overline{H}^s}\sim \|w_0\|_{H^s} $.
It thus  follows from  \eqref{wF1} -\eqref{estdiffHscritical} and dilation arguments
 that for $ T\sim  (1+ \|u_0\|_{H^{1-\frac{\alpha}{2}}})^{-\frac{2(\alpha+1)}{2\alpha-1}}$ and any  $ -\frac{1}{2} \le r\le s $  it holds
\begin{equation}\label{to11}
\|w\|_{Z_T^{r}}\lesssim \|w(0)\|_{H^{r}} \lesssim N_1^{r-s}\varepsilon(N_1)
\end{equation}
with $ \varepsilon(y) \to 0 $ as $ y\to +\infty $ . Moreover, on account of  Lemma \ref{YY} and Proposition \ref{YYY}, for any $ r\ge 0 $ we have
\begin{equation}\label{to22}
\|u^{N_i}\|_{Y_T^{s+r}}
\lesssim
\|u^{N_i}_0\|_{H^{s+r}}
\lesssim N_i^{r} \|u_0\|_{H^s} \, .
\end{equation}
Next, observing that $w$ solves the equation
\begin{equation}\label{WW}
w_t+L_{\alpha +1} w =\frac{1}{2} \partial_x (w^2) +\partial_x (u^{N_1} w) \, ,
\end{equation}
we derive the following estimate on $ w$.
\begin{proposition}
Let $1<\alpha\le 2 $, $ 0<T<2$ and $ w\in L^\infty_T H^s $ with $ s\ge 1-\frac{\alpha}{2}$ be a solution to \eqref{WW}. Then it holds
\begin{eqnarray}
\|w\|_{L^\infty_TH^s}^2  &\lesssim & \|w_0\|_{H^s} ^2 +\|w\|_{Z^s_T}^3  \nonumber \\
& & +  \|u^{N_1} \|_{Y_T^s} \|w\|_{Z_T^s}^2
  + \|u^{N_1} \|_{Y_T^{s+\frac{3}{2}-\frac{\alpha}{2}}}  \|w\|_{Z_T^{s-\frac{3}{2}+\frac{\alpha}{2}}} \|w\|_{Z_T^s}
\;.
\label{estWW}
\end{eqnarray}
\end{proposition}
\begin{proof}
 We separate the contribution of $ \partial_x(w^2) $ and $ \partial_x (u^{N_1} w) $. First   \eqref{N} leads to
 $$
 \sum_{N}N^{2s}  \Bigl| \int_0^t \int_{\R} P_N \partial_x(w^2) P_N w \Bigr| \lesssim \| w\|_{Y_T^s}^3  \; .
 $$
 Second, applying \eqref{NN} at the level $s $ with $z$ replaced by $ u^{N_1} $  we obtain
 $$
  \sum_{N} N^{2s} \Bigl| \int_0^t \int_{\R} P_N \partial_x(u^{N_1} w) P_N w \Bigr| \lesssim \|u^{N_1} \|_{Y_T^{s+\frac{3}{2}-\frac{\alpha}{2}}}
\|w\|_{Z^{-\frac{1}{2}}_T}    \|w\|_{Z_T^s}
  + \|u^{N_1} \|_{Y_T^{1-\frac{\alpha}{2}}} \|w\|_{Z_T^s}^2
 $$
which   leads to \eqref{estWW} since $ s-\frac{3}{2}+\frac{\alpha}{2} \ge -1/2$ for $ s\ge 1-\frac{\alpha}{2}$ and $ Z^s_T \hookrightarrow Y^s_T $.
\end{proof}
 Estimates \eqref{to11}-\eqref{to22} together with \eqref{estWW} lead to
\begin{eqnarray*}
\|w\|_{L^\infty_TH^s}^2  &\lesssim & \|w_0\|_{H^s} ^2 +\varepsilon(N_1)
+ N_1^{-\frac{3}{2}+\frac{\alpha}{2}} \,  N^{-(-\frac{3}{2}+\frac{\alpha}{2})}_1 \varepsilon(N_1)  \label{zq2}\\
 & \lesssim & \varepsilon(N_1) \nonumber    \; .
\end{eqnarray*}
This shows that $ \{u^N\} $ is a Cauchy sequence in $C([0,T];H^s) $ and thus $ \{u^N\} $ converges in $C([0,T];H^s) $ to a solution of \eqref{bo} emanating from $ u_0$. Therefore, the uniqueness result ensures that $ u\in C([0,T];H^s)$. The proof of the  continuity of the solution-map is exactly the same as
 in Subsection \ref{cont} as thus will be omitted .

\subsection{The periodic case}
We notice that all our estimates still hold in the periodic case and are uniform with respect to the period $ L \ge 1$  as soon as only  frequencies with modulus greater or equal to $ \frac{2\pi}{L} $ are involved.  We thus have only  to care about
 the contribution of the null frequencies. In the regular case, it is not too hard to check that all the estimates still hold when we also consider the contribution of the null frequencies. This is because we only use the resonance hypothesis \eqref{hyp2} for high input frequencies (see Remark \ref{remark2}).
In the non regular case, this is no longer true. Anyway, it is easy to check that  \eqref{bo}  preserves the mean-value and
it is well-known that the map $ u \mapsto v(t,x):=u(t,x-t{\int\hspace*{-3mm}-} u_0) -{\int\hspace*{-3mm}-} u_0 $ maps a solution of \eqref{bo} with mean-value
${\int\hspace*{-3mm}-} u_0 $ to a solution of \eqref{bo} with mean value zero. Therefore, up to this change of unknown, we may always assume that our solutions have mean-value zero and thus all the estimates still hold in the periodic setting.
 The proof of Theorem \ref{theo1} is  now completed.
 \section{Dissipative limits}\label{section5}
First, we notice that if $ u $ is solution to \eqref{bob}$_\varepsilon $ then $ u_\lambda $ defined by
 $ u_\lambda(t,x)=\lambda^{\alpha} u(\lambda^{1+\alpha} t, \lambda x) $ is solution to
 \begin{equation}\label{to111}
\partial_t u_\lambda +L_{\alpha+1}^\lambda u_\lambda +\varepsilon\lambda^{\alpha+1-\beta}  A_{\beta}^{\lambda} \, u_\lambda  +\frac{1}{2} \partial_x
  (u_\lambda)^2 =0 \,
 \end{equation}
  with
  $$
  \widehat{L_{\alpha+1}^\lambda v}(\xi)=i\lambda^{\alpha+1} p_{\alpha+1}(\lambda^{-1} \xi)\hat{v}(\xi) \, .
  $$
and
$$
  \widehat{ A_{\beta}^{\lambda} v}(\xi)=\lambda^{\beta} q_{\beta}( \lambda^{-1}\xi)\hat{v}(\xi), \; \forall \xi \in \R \, .
  $$
  Therefore, as in the preceding section, up to this change of unknown, of parameter $ \varepsilon $  and of operators we may assume that $ u $ satisfies
   \eqref{bob} with $ L_{\alpha+1} $ and $ A_\beta $ that verify Hypotheses \ref{hyp1} and \ref{hypdis}  for all $ 0<\lambda\le 1$. \\
Second, we notice  that Hypothesis \ref{hypdis} now ensures that for $ 0<\lambda \le 1 $ and $ N\gg 1 $ dyadic,
 \begin{equation}\label{A1}
 (A_{\beta}^{\lambda} P_N v, P_N v)_{L^2} \gtrsim N^{\beta/2} \|P_N v\|_{L^2}^2
 \end{equation}
 and
 \begin{equation}\label{A2}
 \|A_{\beta}^{\lambda} P_N v \|_{L^2} \lesssim N^{\beta} \|P_N v\|_{L^2} \, .
 \end{equation}
 The main point is now to prove that the Cauchy problem \eqref{bob} is locally well-posed in $ H^s $ uniformly  in $ \varepsilon>0 $.
   \begin{proposition}\label{prop2}
Let $ 1\le \alpha\le 2,\,  0\le \beta \le 1+\alpha $ and $ s\ge 1-\frac{\alpha}{2}$.\\
 For any   $ \varphi\in H^s(\R) $ there exists $ T\sim (1+\|u_0\|_{H^{1-\frac{\alpha}{2}}})^{-\frac{2(\alpha+1)}{2\alpha-1}}  $ and a  solution $ u_\varepsilon\in  C([0,T]; H^{s}) $  to \eqref{bob}$_\varepsilon$ that is unique in some function space\footnote{For  $(\alpha,s)\neq (1,1/2)$, this space is simply the space $ L^\infty_T H^s\cap L^2_T H^{s+\frac{\beta}{2}}$} embedded in $L^\infty_T(0,T; H^s)$.  Moreover, there exists  $ C>0 $ such that for any $  \varepsilon\in ]0,1[ $,  \begin{equation}\label{to3}
 \sup_{t\in [0,T]} \| u_\varepsilon (t)\|_{H^s} \le C \|\varphi\|_{H^s} \; .
\end{equation}
Finally, for any $ R>0 $, the family of  solution-maps $S_{\varepsilon}  \, :\,  \varphi \mapsto u_\varepsilon $, $ \varepsilon\in ]0,1[ $, from $ B(0,R)_{H^s}  $ into $  C([0,T(R)]; H^s(\R)) $ is equi-continuous, i.e. for any sequence $ \{\varphi_n\} \subset   B(0,R)_{H^s}  $ converging to $ \varphi $ in $ H^s(\R) $ it holds \begin{equation}\label{to4}
\lim_{n\to +\infty}  \sup_{\varepsilon\in ]0,1[}\|S_{\varepsilon} \varphi - S_{\varepsilon} \varphi_n\|_{L^\infty(0,T(R); H^s(\R))}=0 \; .
 \end{equation}
  \end{proposition}
\begin{proof} We treat the cases $ (\alpha,s) \neq (1,1/2) $. This last case can be treated in the same way by using the estimates derived in the appendix.
First we notice that for \eqref{bob}$_\varepsilon $, in view of \eqref{A1}, the energy estimate \eqref{estHsnonregular} becomes
\begin{equation}
\|u\|_{\widetilde{L^\infty_T}H^s}  +\sqrt{\varepsilon} \|u\|_{L^2_T H^{s+\frac{\beta}{2}}} \lesssim \|u_0\|_{H^s}  +
 \|u\|_{L^\infty_T H^{1-\frac{\alpha}{2}}} \|u\|_{Y^s_T} + \|u\|_{L^\infty_T H^s}\|u\|_{Y_T^{1-\frac{\alpha}{2}}}
 \;.
\end{equation}
On the other hand, viewing $ \varepsilon A_{\beta} u  $ as a forced term, \eqref{Y1}-\eqref{Y2} together with \eqref{A2} lead to
\begin{equation}\label{Y3}
\|u\|_{Y^{s}_T}\lesssim \|u\|_{L^\infty_T H^s} (1+  \|u\|_{L^\infty_T H^{1-\frac{\alpha}{2}}}^2) + \varepsilon \|u\|_{L^2_T H^{s-\frac{1+\alpha}{2}+\beta}}
\, .
\end{equation}
To derive an a priori bound from the above estimates, as in the previous section,  we have to use  the dilation argument  that is described in the beginning of this section. So the dilation function $ u_\lambda $, defined by
 $ u_\lambda(t,x)=\lambda^{\alpha} u(\lambda^{1+\alpha} t, \lambda x) $, satisfies \eqref{to111} and we set
  $$
  \|v\|_{N^s}:=  \|v\|_{L^\infty_TH^s}  +\sqrt{\varepsilon \lambda^{\alpha+1-\beta}} \|v\|_{L^2_T H^{s+\frac{\beta}{2}}} \; .
  $$
  Since $ \beta \le \alpha+1 $, this ensures that for $ \lambda \lesssim (1+\|\varphi\|_{H^s})^{-\frac{2(\alpha+1)}{2\alpha-1}} $ and $ 0<T\le 2$,  it holds
  \begin{eqnarray*}
  \|u_\lambda\|_{N^s_T}& \lesssim &
   \|\varphi_\lambda\|_{H^s}  +(1+\|u_\lambda\|_{N^{1-\frac{\alpha}{2}}_T}^2) \|u_\lambda\|_{N^{1-\frac{\alpha}{2}}_T}\|u_\lambda\|_{N^s_T}
 \;.
  \end{eqnarray*}
  with $   \|\varphi_\lambda\|_{H^s}\lesssim \lambda^{\alpha-\frac{1}{2}} \|\varphi\|_{H^s} \ll 1 $. This leads to  the uniform bound \eqref{to3} for smooth solutions
   to \eqref{bob}$_\varepsilon $ by a classical continuity argument. Then passing to the limit on sequence of smooth solutions
   we obtain the existence of a solution
  $ u_\varepsilon \in L^\infty_T H^s\cap L^2_T H^{s+\frac{\beta}{2}} $ to \eqref{bob}$_\varepsilon$ with  $  T\gtrsim (1+\|u_0\|_{H^{1-\frac{\alpha}{2}}})^{-\frac{2(\alpha+1)}{2\alpha-1}} $ and $ \varphi\in H^s $ as initial data. Obviously, this solution satisfies \eqref{to3}.

  Now, proceeding in the same way for the difference of two solutions, it is not too hard to check that \eqref{wF1} becomes
  \begin{align*}
\|u-v\|_{Z^{s-\frac{3}{2}+\frac{\alpha}{2}}_T} \lesssim  & \|u-v\|_{\widetilde{L^\infty_T} \overline{H}^{s-\frac{3}{2}+\frac{\alpha}{2}}}
 +   \|u-v\|_{L^2_T {H}^{s-\frac{3}{2}+\frac{\alpha}{2}+\beta}}\\
 & + \|u+v\|_{\widetilde{Y}^s_T} \|u-v\|_{Z^{-1/2}_T}
+ \|u+v\|_{\widetilde{Y}^{1-\frac{\alpha}{2}}}\|u-v\|_{Z^{s-\frac{3}{2}+\frac{\alpha}{2}}_T}\, .
  \end{align*}
  whereas  \eqref{estdiffHscritical} becomes
 \begin{align*}
\|u_\varepsilon-v_\varepsilon\|_{\widetilde{L^{\infty}_T}\overline{H}^{s-\frac{3}{2}+\frac{\alpha}{2}} }  + &
\sqrt{\varepsilon} \|u_\varepsilon-v_\varepsilon\|_{L^{2}_T\overline{H}^{s-\frac{3}{2}+\frac{\alpha}{2}+\frac{\beta}{2}} } \\
& \lesssim \|u_0-v_0\|_{\overline{H}^{s-\frac{3}{2}+\frac{\alpha}{2}} }  +
\|u_\varepsilon+v_\varepsilon\|_{\widetilde{Y}^s_T}
\|u_\varepsilon-v_\varepsilon\|_{Z^{s-\frac{3}{2}+\frac{\alpha}{2}}_T} \;.
\end{align*}
  By the same dilation arguments as above this leads to
  \begin{equation}\label{sz}
  \|u-v\|_{Z^{s-\frac{3}{2}+\frac{\alpha}{2}}} + \sqrt{\varepsilon}  \|u-v\|_{L^2_T H^{s-\frac{3}{2}+\frac{\alpha}{2}+\frac{\beta}{2}} }
  \lesssim \|u_0-v_0\|_{\overline{H}^{s-\frac{3}{2}+\frac{\alpha}{2}}}
  \end{equation}
  and proves the uniqueness in the class $L^\infty_{loc} H^s \cap L^2_{loc} H^{s+\beta/2}$. Finally the continuity of the solution and the equi-continuity of the solution-map in $ C(0,T;H^s) $ follows from Bona-Smith arguments as in the previous section.
\end{proof}
It is clear that the above proposition implies  part $(1)$ of Theorem \ref{theo2}.
 Now,  part $ (2) $ will follow from general arguments (see for instance  \cite{GW}).   Let us denote by $ S_\varepsilon $ and $ S$ the nonlinear
  group associated with respectively \eqref{bob}$_\varepsilon $   and \eqref{bo}. Let  $ \varphi\in H^s(\R) $, $s\ge 1-\frac{\alpha}{2}$ and let $ T=T(\|\varphi\|_{H^{1-\frac{\alpha}{2}}} ) >0 $ be given by Proposition \ref{prop2}. For any $ N>0 $  we can rewrite $ S_{\varepsilon}(\varphi)-
  S(\varphi) $ as
  \begin{eqnarray*}
  S_{\varepsilon}(\varphi)-S(\varphi)&=& \Bigl(S_{\varepsilon}(\varphi)-S_{\varepsilon}(P_{\le N} \varphi)\Bigr)+
  \Bigl(S_{\varepsilon}(P_{\le N} \varphi)- S(P_{\le N}\varphi)\Bigr)\\
  & & + \Bigl(S(P_{\le N}\varphi)-S(\varphi)\Bigr)=I_{\varepsilon,N}+J_{\varepsilon,N}+K_{N}\; .
  \end{eqnarray*}
  By continuity with respect to initial data in $ H^s(\R) $ of the solution map associated with \eqref{bo}, we have  $\displaystyle \lim_{N\to \infty}
  \|K_N\|_{L^\infty(0,T;H^s) }=0 $. Moreover,  \eqref{to4} ensures that
  $$
  \lim_{N\to \infty} \sup_{\varepsilon\in]0,1[}\|I_{\varepsilon,N}\|_{L^\infty(0,T;H^s)} =0 \; .
  $$
  It thus remains to check that for any fixed $ N>0 $,  $ \displaystyle \lim_{\varepsilon\to 0} \|J_{\varepsilon,N}\|_{L^\infty(0,T;H^s_x)}=0 $.
  Since $ P_{\le N} \varphi\in H^\infty(\R) $, it is worth noticing that
   $ S_{\varepsilon}(P_{\le N} \varphi) $ and $ S(P_{\le N}\varphi) $ belong to $C^\infty (\R;H^\infty(\R)) $. Moreover, according to Theorem \ref{theo2} and  Proposition \ref{prop2},  for all $ \theta\in \R$ and $ \varepsilon\in ]0,1[ $,
   $$
   \|S_{\varepsilon}(P_{\le N} \varphi) \|_{L^\infty_T H^{\theta}_x}+ \|S(P_{\le N} \varphi)\|_{L^\infty_T H^{\theta}_x} \le C(N,\theta,  \|\varphi\|_{L^2_x})  \; .
   $$

 Now, setting   $ v_\varepsilon := S_{\varepsilon}(P_{\le N} \varphi)  $ and  $v:=S(P_{\le N}\varphi) $, we  observe that   $ w_\varepsilon:= v_\varepsilon-v $ satisfies
   $$
    \partial_t w_\varepsilon+L_{\alpha+1} w_\varepsilon = -\frac{1}{2}\partial_x \Bigl( w_\varepsilon
    (v +v_\varepsilon)\Bigr) -\varepsilon A_\beta v_\varepsilon
   $$
   with initial data $ w_\varepsilon(0)=0 $. For $ s\ge 0$, taking the  $ H^s $-scalar product of this last equation with $w_\varepsilon $ and  integrating by parts we get
   $$
   \frac{d}{dt} \|w_\varepsilon\|_{H^s} \lesssim \Bigl( 1+\|\partial_x(v +v_\varepsilon)\|_{L^\infty_x}\Bigl) \| w_\varepsilon\|_{H^s}^2
   + \|[J^s\partial_x , (v +v_\varepsilon)] w_\varepsilon\|_{L^2_x} \| w_\varepsilon\|_{H^s} + \varepsilon^2 \|D_x^{\beta}v_\varepsilon\|_{H^s}^2 \; .
   $$
 Applying  the mean-value theorem  on the Fourier transform of the commutator term, it is not too hard to check that
   \begin{equation}\label{commute}
   \| [J^s_x\partial_x , f] g \|_{L^2_x} \lesssim \|f_x\|_{H^{s+1} }\| g\|_{H^s_x}\; ,
   \end{equation}
   that leads to
   $$
   \frac{d}{dt} \| w_\varepsilon(t)\|_{H^s}^2\lesssim C(N,s+2,  \|\varphi\|_{L^2_x})
   \| w_\varepsilon(t)\|_{H^s_x}^2+ \varepsilon^2  C(N,s+\beta,  \|\varphi\|_{L^2_x})^2\; .
   $$
   Integrating this differential inequality on $ [0,T]$, this ensures  that $  \displaystyle \lim_{\varepsilon\to 0} \| w_\varepsilon\|_{L^\infty(0,T;H^s) }=0 $
    and proves that
\begin{equation}
\label{j1}
    u_\varepsilon \longrightarrow u \; \mbox{ in } C([0,T]; H^s)
\end{equation}
    with $ T\sim  (1+\|u_0\|_{H^{1-\frac{\alpha}{2}}})^{-\frac{2(\alpha+1)}{2\alpha-1}} $.
Now, the fact that, $ \varphi $ being fixed,  the  time of existence $ T_\varepsilon $ of $ S_\varepsilon (\varphi) $ in $ H^s $ is greater or equal
   for $ \varepsilon>0 $ small enough  to the time of existence $ T_0 $ of
 $ S (\varphi) $  follows by a classical contradiction argument. Indeed, assuming that this is not true, there exists $ \varepsilon_n \searrow 0 $ such that $\lim T_{\varepsilon_n}=T^*<T_0 $. We set
$$ \delta= (1+\|S(\varphi)\|_{L^\infty(0,T^*;H^{1-\frac{\alpha}{2}})})^{-\frac{2(\alpha+1)}{2\alpha-1}}
$$
which is  well-defined since $ T^*<T $.  Applying \eqref{j1} about $ T^*/\delta $ times we eventually obtain that for $ n $ large enough
$$
\|S_{\varepsilon_n}(\varphi)(T^*-\frac{\delta}{100}) \|_{H^{1-\frac{\alpha}{2}}}\le 2
\|S(\varphi)\|_{L^\infty(0,T^*;H^{1-\frac{\alpha}{2}})}\; .
 $$
 But then the uniform bound from below on the existence time ensures that $ T_{\varepsilon_n} \ge T^*+ \delta/2 $ that contradicts  $\lim T_{\varepsilon_n}=T^*$. This ensures that $ T_\varepsilon\ge T_0 $ for $ \varepsilon>0 $ small enough and, for $0<T^*<T_0$, applying \eqref{j1} about $ T^*/\delta $ times we get \eqref{j1} with $T=T^*$. This completes the proof of Theorem \ref{theo2}.
  \section{Appendix: The case $ \alpha=1 $ and $ s=1/2$.}
 This case is important since $ H^{1/2}$ is the energy space for the Benjamin-Ono equation and also the Intermediate Long Waves equation. Unfortunately, we are not able to prove the unconditional well-posedness in this case. However, we are able to prove the well-posedness without using a gauge transform. This is  useful to treat perturbations of these equations  as we explained in the preceding section. In this section we indicate the modifications of the proofs in this case.
 In the sequel we set
 $$ \widetilde{M}^{1/2}:= \widetilde{L^\infty_t} H^{1/2} \cap X^{-1/2,1} \; .
 $$
 \begin{lemma} \label{lemA1} Let  $ \alpha=1 $, $ 0<T<2$, and $ u\in \widetilde{M}_T^{1/2} $  be a solution to \eqref{bo}. Then it holds
 \begin{equation}
 \|u\|_{ \widetilde{M}^{1/2}_T} \lesssim \|u\|_{\widetilde{L^\infty_T} H^{1/2}} + \|u\|_{ \widetilde{M}^{1/2}_T}^2 \; .
 \end{equation}
 \end{lemma}
 \begin{proof}
 Working with the extension $\tilde{u}=\rho_Tu$ (see (\ref{defrho})), still denoted $u$, if suffices to estimate the $X^{-1/2,1}$-norm of $u$.
 First we notice that the low frequency part can be easily controlled by
 $$
\| P_{\le2^9} u  \|_{X^{-1/2,1}_T} \lesssim \|u\|_{L^\infty_T L^2_x}^2 \; .
$$
Now for $ N\ge 2^9 $, we have
 \begin{eqnarray*}
 \| u_N\|_{X^{-1/2,1}_T} &\lesssim& \|P_Nu_0\|_{H^{-1/2}} +  N^{1/2} \left\|\sum_{N_2'\sim N_2\ge N} u_{N_2} u_{N_2'} \right\|_{L^2_{T}L^2_x}
 \\ & &+ N^{1/2}  \left\|\sum_{N_2\ll N} P_N( \ut_{\sim N}  \, u_{N_2}) \right\|_{L^2_{T}L^2_x} \\
 & = & \|P_Nu_0\|_{H^{-1/2}} + I_N + I\! I_N \; .
 \end{eqnarray*}
 Clearly, it holds
  \begin{eqnarray*}
I_N  & \lesssim &  N^{1/2}\sum_{N_2'\sim N_2\ge N}  \| u_{N_2} \|_{L^2_t H^{1/2}}  \| | u_{N_2'} \|_{L^\infty_t H^{1/2}} \\
 & \lesssim & \delta_N \| | u\|_{L^\infty_t H^{1/2}}^2  \; ,
 \end{eqnarray*}
 with $ \| (\delta_N)\|_{l^2} \lesssim 1$.
 On the other hand,
   \begin{eqnarray*}
I\! I_N  & \lesssim &  N^{1/2}\Bigl\|\sum_{N_2 \ll  N} Q_{\sim N N_2} P_N(  u_{\sim N} \,  u_{N_2})\Bigr\|_{L^2_{tx}}
+ N^{1/2}\Bigl\|\sum_{N_2 \ll  N} Q_{\not \sim N N_2} P_N( u_{\sim N}  \,  u_{N_2})\Bigr\|_{L^2_{tx}} \\
 & \lesssim &  I\! I^1_N + I\! I^2_N \; .
 \end{eqnarray*}
 By almost orthogonality, we have
   \begin{eqnarray*}
I\! I^1_N  & \lesssim &  N^{1/2}\Bigl(\sum_{N_2 \ll  N}  \Bigl\| Q_{\sim N N_2} P_N( u_{\sim N} \,   u_{N_2})\Bigr\|_{L^2_{tx}}^2 \Bigr)^{1/2} \\
 &\lesssim & N^{1/2} \Bigl( \sum_{N_2 \ll  N} \| u_{\sim N} \|_{L^2_t  H^{1/2}}^2 \|  u_{N_2}\|_{L^\infty_t L^2_x}^2 \Bigr)^{1/2} \\
 & \lesssim &   \| u_{\sim N} \|_{L^2_t  H^{1/2}} \| u\|_{\widetilde{L^\infty_t}  H^{1/2}} \\
 & \lesssim & \delta_N  \|u\|_{L^\infty_t  H^{1/2}} \| u\|_{\widetilde{L^\infty_t}  H^{1/2}} \; ,
 \end{eqnarray*}
 with $ \| (\delta_N)\|_{l^2} \lesssim 1$.
 It remains to control $ I\! I^2_N $. Since the Fourier projectors ensure that $ \langle \tau-p_2(\xi) \rangle \not\sim N N_2 $, the resonance relation \eqref{hyp2} leads to $ |\tau_1-p_2(\xi_1)|\vee |\tau-\tau_1-p_2(\xi-\xi_1)|) \ge N N_2 $  for $ I\! I^2_N$.
 We separate the contributions of $ Q_{\ge N N_2}  \ut_{\sim N}$ and  $Q_{\ge N N_2}  \ut_{N_2}$.
For the first contribution we  have
\begin{eqnarray*}
I\! I_{N}^2 & \lesssim & N^{1/2} \sum_{N_2\ll N} (N N_2)^{-1/4} N^{1/4} \|Q_{\ge NN_2} u_{\sim N} \|_{X^{1/4,1/4}} \| u_{N_2}\|_{L^\infty_t H^{1/2}} \\
& \lesssim &  \| u_{\sim N} \|_{X^{1/4,1/4}} \| u\|_{L^\infty_t H^{1/2}} \\
& \lesssim & \delta_N \|u\|_{X^{-1/2,1}}^{1/4} \| u\|_{L^\infty_t H^{1/2}}^{3/4} \| u\|_{L^\infty_t H^{1/2}} \, ,
\end{eqnarray*}
with $ \| (\delta_N)\|_{l^2} \lesssim 1$ and where we used interpolation at the last step.  For the second contribution we write
\begin{eqnarray*}
I\! I_{N}^2 & \lesssim & N^{1/2} \sum_{N_2\ll N} \|Q_{<N N_2}  u_{\sim N}  \|_{L^\infty_t L^4_x} \|Q_{\ge N N_2}  u_{N_2}\|_{L^2_t L^4_x} \\
& \lesssim & N^{1/2} \sum_{N_2\ll N} N^{-1/4}\|Q_{<N N_2}  u_{\sim N}  \|_{L^\infty_t H^{1/2}} \|Q_{\ge N N_2}  u_{N_2}\|_{L^2_t H^{1/4}} \\
& \lesssim & N^{1/2} \sum_{N_2\ll N} N^{-1/4} (N N_2)^{-1/4} \| u_{\sim N}  \|_{L^\infty_t H^{1/2}} \|  u_{N_2}\|_{X^{1/4,1/4}} \\
& \lesssim & \delta_N  \| u\|_{\widetilde{L^\infty_t} H^{1/2}}\| u\|_{X^{-1/2,1}}^{1/4} \|| u\|_{L^\infty_t H^{1/2}}^{3/4}  \, ,
\end{eqnarray*}
with $ \| (\delta_N)\|_{l^2} \lesssim 1$.
Gathering the above estimates, \eqref{A1} follows.
\end{proof}
 \begin{lemma} \label{lemA2} Let  $ \alpha=1 $, $ 0<T<2$ and $ u\in \widetilde{M}_T^{1/2} $  be a solution to \eqref{bo}. Then it holds
 \begin{equation}
 \|u\|_{\widetilde{L^\infty_T} H^{1/2}} \lesssim \|u\|_{\widetilde{L^\infty_t} H^{1/2}} +  \|u\|_{\widetilde{L^\infty_t} H^{1/2}} \|u\|_{ \widetilde{M}^{1/2}_T} \; .
 \end{equation}
 \end{lemma}
 \begin{proof}
 We follow the proof of Proposition \ref{YYY}. Note that $ \widetilde{M}^{1/2} \hookrightarrow \widetilde{Y}^{1/2} $. According to \eqref{415} it suffices to control
  $$
I = \sum_{N>0}\sum_{N_1\gtrsim  N}  N\langle N_1\rangle\sup_{t\in ]0,T[} |I_t(u_N, u_{\sim N_1}, u_{N_1})|.
$$
It is easy to check that the only term of the left-hand side of  \eqref{416}  that causes trouble in the case $\alpha=1 $ is the first one. This term corresponds to the contribution of $ Q_{2^l N N_1^\alpha} u_{N_1} $ and $ Q_{\sim N N_1^{\alpha}} u_{N_1} $.  For $ \alpha=1 $ we control these contributions by applying Cauchy-Schwarz in
 $(N,N_1) $. For instance, the contribution of $ Q_{2^l N N_1^\alpha} u_{N_1} $ is estimated  thanks to Lemma \ref{lemtriestY} by 
 \begin{align*}
 \sum_{N>2^9} &\sum_{N_1\gtrsim N}N\langle N_1 \rangle \sum_{l\ge -4} 2^{-l} N^{-1/2}  \|u_N\|_{L^2_{tx}} \|Q_{2^l N N_1^\alpha} u_{N_1}\|_{F^{0,1/2}}
  \|u_{\sim N_1} \|_{L^\infty_t L^2_x} \\
& \lesssim \sum_{l\ge -4} 2^{-l} \Bigl( \sum_{N_1\gtrsim N\ge 2^9} \|u_N\|_{L^2_t H^{1/2}}^2 \| u_{\sim N_1}\|_{L^\infty_t H^{1/2}}^2 \Bigr)^{1/2}\Bigl( \sum_{N_1\gtrsim N\ge 2^9} \| Q_{2^l N N_1^\alpha} u_{N_1} \|_{F^{1/2,1/2}}^2  \Bigr)^{1/2}\\
  &\lesssim \|u\|_{L^2_t H^{1/2}} \|u\|_{\widetilde{L^\infty_t} H^{1/2}} \|u\|_{F^{1/2,1/2}} \; .
 \end{align*}
 \end{proof}
  \begin{lemma} \label{lemA3} Let $ 0<T<1$  and $u,v \in \widetilde{M}^{1/2}_T $  be two solutions to (\ref{bo}) on $]0,T[$. Then we have
\begin{equation}\label{A9}
\|u-v\|_{Z^{-\frac{1}{2}}_T} \lesssim \|u-v\|_{L^\infty_T \overline{H}^{s-\frac{3}{2}+\frac{\alpha}{2}}} + \|u+v\|_{ \widetilde{M}^{1/2}_T} \|u-v\|_{Z^{-1/2}_T}
+ \|u+v\|_{ \widetilde{M}^{1/2}_T}\|u-v\|_{Z^{-\frac{1}{2}}_T}\, .
\end{equation}
and
\begin{equation}\label{A10}
\|u-v\|_{\widetilde{L^{\infty}_T}\overline{H}^{-\frac{1}{2} }}  \lesssim \|u_0-v_0\|_{\overline{H}^{-\frac{1}{2}} }  +
\|u+v\|_{\widetilde{M}^{1/2}_T}
\|u-v\|_{Z^{-\frac{1}{2}}_T} \;.
\end{equation}
 \end{lemma}
 \begin{remark}
 Actually we could avoid to put a   weight on the low frequencies of the difference $u-v$.  However, working in $ Z^{-1/2}_T $ allow us to use directly some results of Section \ref{section4}.
 \end{remark}
 \begin{proof}
 First we notice that \eqref{A10} is already proven in Proposition \ref{PP1}. since $ \widetilde{M}^{1/2}_T   \hookrightarrow  \widetilde{Y}^{1/2}_T\hookrightarrow Y^{1/2}_T $. It remains to prove \eqref{A9}.
 We follow the proof of Proposition \ref{YYY}. It is not to hard to check that the only contribution that causes troubles in the right-hand side member of
  \eqref{est-wF}, in the case $ \alpha=1 $, is the contribution of the low-high interaction term : $P_N (P_{\lesssim N} z \, w_N) $.
  We proceed as in Lemma \ref{lemA1} . We take extensions   $\widetilde{z} $ and $ \widetilde{w} $, supported in $ ]-4,4[ $  of $ z $ and $ w$ such that $
  \|\widetilde{z}\|_{\widetilde{M}^{1/2}}\lesssim \|z\|_{\widetilde{M}^{1/2}_T} $ and $ \|\widetilde{w} \|_{Z^{-1/2}}\lesssim \|w\|_{Z^{-1/2}_T} $. For simplicity we drop the tilde. We first notice that the contribution of $ P_{\lesssim 1} z  $ is easily estimated by
  $$
  \|\partial_xP_N(P_{\lesssim 1}z  \, w_{\sim N} )\|_{F^{-1/2,1,-1/2}} \lesssim \langle N\rangle^{-1/2} \|    P_N(P_{\lesssim 1}z  \, w_{\sim N})\|_{L^2_{tx}}
  \lesssim \|z \|_{L^\infty_t L^2_x}  \|w_{\sim N} \|_{L^2_t H^{-1/2}}
  $$
  which is acceptable. Now we decompose the remaining contribution as
 \begin{align*}
\|\partial_xP_N(P_{\gg 1} P_{\lesssim N}z w_{\sim N})\|_{F^{-1/2,1,-1/2}} &\lesssim N\|   \sum_{ 1\ll  N_1\lesssim N} P_N(P_{N_1}z w_{\sim N})\|_{X^{-3/2,0}}\\
& \lesssim \langle N\rangle^{-1/2} \|   \sum_{1\ll N_1\lesssim N} Q_{\sim N N_1} P_N(P_{N_1}z  \, w_{\sim N})\|_{L^2_{tx}} \\
& + \langle N\rangle^{-1/2}\|   \sum_{1\ll N_1\lesssim N} Q_{\not \sim N N_1} P_N(P_{N_1}z  \, w_{\sim N})\|_{L^2_{tx}}\\
& =  J_{1,N}+J_{2,N} \; .
\end{align*}
By almost-orthogonality it holds
 \begin{align*}
J_{1,N} & \lesssim \langle N\rangle^{-1/2}  \Bigl(  \sum_{1\ll N_1\lesssim N} \|Q_{\sim N N_1} P_N(P_{N_1}z  \, w_{\sim N})\|_{L^2_{tx}}^2 \Bigr)^{1/2}\\
& \lesssim \langle N\rangle^{-1/2} \Bigl( \sum_{1\ll N_1\lesssim N} \|P_{N_1}z \|_{L^2_t H^{1/2}}^2 \| \, w_{\sim N}\|_{L^\infty_{tx}}^2 \Bigr)^{1/2}\\
& \lesssim \|  w_{\sim N}\|_{L^\infty_t H^{-1/2}} \|z\|_{L^2_t H^{1/2}} ,
\end{align*}
which is acceptable.
To treat $ J_2 $, we notice that the Fourier projectors ensure that $ \langle \tau-p_2(\xi) \rangle \not\sim N N_1 $, the resonance relation \eqref{hyp2} leads to $ |\tau_1-p_2(\xi_1)|\vee |\tau-\tau_1-p_2(\xi-\xi_1)|) \ge N N_1 $  for $J_{2,N}$.
 We separate the contributions of $ Q_{\ge N N_1}  \zt_{N_1}$ and  $Q_{\ge N N_1}  \wt_{\sim N}$.
For the first contribution we  write
\begin{eqnarray*}
J_{2,N} & \lesssim & \langle N\rangle^{-1/2}  \sum_{1\ll N_1\lesssim N}  N_1^{1/2} \|Q_{\ge NN_1}P_{N_1} z\|_{L^2_{tx}} \|  w_{\sim N}\|_{L^\infty_t L^2_x} \\
& \lesssim & \langle N\rangle^{-1/2}  \sum_{1\ll N_1\lesssim N} (N N_1)^{-1/4} N_1^{1/4} \|Q_{\ge NN_1}P_{N_1} z\|_{X^{1/4,1/4}}
 \|  w_{\sim N}\|_{L^\infty_t L^2_x} \\
& \lesssim &  \|z\|_{X^{-1/2,1}}^{-1/4} \|z\|_{L^\infty_t H^{1/2}}^{3/4} \|  w_{\sim N}\|_{L^\infty_t H^{-1/2}} \, ,
\end{eqnarray*}
which is acceptable.  For the second contribution, according to \eqref{B1}  we have
\begin{eqnarray*}
J_2
& \lesssim & \langle N\rangle^{-1/2}  \sum_{1\ll N_1\lesssim  N}
\| z_{N_1} \|_{L^\infty_t H^{1/2}} \|Q_{\ge N N_1}   w_{\sim N}\|_{L^2_{tx}} \\
& \lesssim & \langle N\rangle^{-1/2}  \sum_{1\ll N_1\lesssim  N} (N N_1)^{-1} N^{3/2}
\| z_{N_1} \|_{L^\infty_t H^{1/2}} \|  w_{\sim N}\|_{F^{-1/2,1/2}} \\
& \lesssim & \|  w_{\sim N}\|_{F^{-1/2,1/2}} \|z\|_{L^\infty_t H^{1/2}}
\end{eqnarray*}
which is acceptable. Gathering the above estimates we obtain \eqref{A9}.
 \end{proof}
Gathering  Lemmas \ref{lemA1}-\ref{lemA3} and proceeding as in Subsection \ref{section42} we obtain the local-well-posedness in $
 H^{1/2} $ of \eqref{bo}  for $ \alpha=1 $. Note that the  uniqueness holds in the space $ \widetilde{M}^{1/2}_T $.
 \bibliographystyle{amsplain}

\end{document}